\documentclass{article}

\usepackage{amsmath}
\usepackage{amsthm}
\usepackage{amssymb}
\usepackage[all,cmtip,knot]{xy}
\usepackage{graphicx}
\usepackage{pinlabel}

\newcommand{\hyperbox}{\mathcal{H}}
\newcommand{\F}{\mathcal{F}}
\newcommand{\C}{\mathcal{C}}
\newcommand{\hfbold}{\mathbf{HF}}
\newcommand{\hfboldinfty}{\mathbf{HF}^\infty}
\newcommand{\fpowers}{\mathbb{F}[[U,U^{-1}]}
\newcommand{\fexterior}{\Lambda^*_{\mathbb{F}}}
\newcommand{\cinfty}{C^\infty}
\newcommand{\hcinfty}{HC^\infty}
\newcommand{\proj}{\mathcal{I}}
\newcommand{\destab}{\mathcal{D}}
\newcommand{\hm}{\overline{HM}}
\newcommand{\boldalpha}{\mbox{\boldmath $\alpha$}}
\newcommand{\boldbeta}{\mbox{\boldmath $\beta$}}

\theoremstyle{plain} \newtheorem{theorem}{Theorem}[section]
\theoremstyle{plain} \newtheorem{proposition}[theorem]{Proposition}
\theoremstyle{plain} \newtheorem{lemma}[theorem]{Lemma}
\theoremstyle{plain} \newtheorem{corollary}[theorem]{Corollary}
\theoremstyle{plain} 
\theoremstyle{definition} \newtheorem{definition}[theorem]{Definition}
\theoremstyle{remark} \newtheorem{remark}[theorem]{Remark}
\theoremstyle{remark} \newtheorem{example}[theorem]{Example}

\title{Heegaard Floer Homology and Triple Cup Products}
\date{}
\author{Tye Lidman}

\begin{document}
\maketitle

\abstract{We give a complete calculation of $HF^\infty(Y,\mathfrak{s})$ with mod 2 coefficients for all three-manifolds $Y$ and torsion Spin$^c$ structures $\mathfrak{s}$.  The computation agrees with the conjectured calculation of Ozsv\'ath and Szab\'o in \cite{plumbed}.  This therefore establishes an isomorphism with Mark's cup homology, $\hcinfty(Y)$, mod 2 \cite{thommark}.

\section{Introduction}

Throughout the last decade, Heegaard Floer homology has been a ubiquitous tool in low-dimensional topology.  Variants of this theory exist for nearly any object one can find a use for and have been studied in a wide range of contexts, such as embeddings of symplectic surfaces in four-manifolds, lens space surgery obstructions, and the classification theory of tight contact structures.  In many cases, these invariants can be calculated completely combinatorially, making them accessible and desirable to utilize.

For closed three-manifolds, the Heegaard Floer chain complexes come in many flavors, $\widehat{CF}, CF^+, CF^-$, and $CF^\infty$.  These flavors are all in fact derived from $CF^\infty$ by some formal construction on the chain level.  Therefore, having an understanding of the homology, $HF^\infty$, provides foundational information for the other flavors; for example, the $\mathbb{Z}$-rank of $\widehat{HF}(Y)$ is always bounded below by the $\mathbb{Z}[U,U^{-1}]$-rank of $HF^\infty(Y)$.

In \cite{hfpa}, Ozsv\'ath and Szab\'o calculate $HF^\infty(Y,\mathfrak{s})$ for all $Y$ with $b_1(Y) \leq 2$ and all Spin$^c$ structures $\mathfrak{s}$.  It is also shown that there is a universal coefficients spectral sequence for torsion Spin$^c$ structures, with $E_3$ term given by $\Lambda^*(H^1(Y;\mathbb{Z})) \otimes \mathbb{Z}[U,U^{-1}]$, which converges to $HF^\infty$.  They conjecture in \cite{plumbed} that the $d_3$ differential is given by contraction via the integral triple cup product form, $\mu_Y$, and that all higher differentials vanish.  This agrees with our previous calculations of $HF^\infty$ extended to $b_1(Y) \leq 4$ for coefficients in $\mathbb{F} = \mathbb{Z}/2\mathbb{Z}$ \cite{hfcoho}.  These computations immediately extend to the case of $\hfboldinfty$, where $\hfboldinfty$ is defined to be $HF^\infty$ with mod 2 coefficients, completed with respect to the variable $U$.  Throughout this paper, all Floer homologies will be calculated with mod 2 coefficients, unless noted otherwise.

In \cite{thommark}, Mark constructs a complex over $\mathbb{Z}[U,U^{-1}]$, $\cinfty_*(Y)$, with chain group $\Lambda^*(H^1(Y;\mathbb{Z})) \otimes \mathbb{Z}[U,U^{-1}]$; in fact, $\cinfty_*(Y)$ is exactly the conjectured complex $(E_3,d_3)$ as mentioned above.  For compatibility, we will work with $\mathbb{F}$ coefficients for $\cinfty_*$ as well, where triple cup products are taken in integral cohomology and then reduced mod 2.

In this paper, we use the link surgery formula of Manolescu and Ozsv\'ath (Theorem 1.1 of \cite{hflz}) to calculate $HF^\infty(Y,\mathfrak{s})$ for all $Y$ and torsion $\mathfrak{s}$ and compare the result with $\hcinfty$.

\begin{theorem} \label{maintheorem}
Let $Y$ be a three-manifold equipped with a torsion Spin$^c$ structure $\mathfrak{s}$.  The relatively-graded $\mathbb{F}[U,U^{-1}]$-modules $HF^\infty_*(Y,\mathfrak{s})$ and $\hcinfty_*(Y)$ are isomorphic.  Therefore, $HF^\infty(Y,\mathfrak{s})$ agrees with Conjecture 4.10 of \cite{plumbed} mod 2.
\end{theorem}

\begin{remark}
It is also known in monopole Floer homology that for torsion Spin$^c$ structures, $\hcinfty(Y) \cong \hm(Y,\mathfrak{s})$ over $\mathbb{Q}$ (see Section IX of \cite{monopolebook}); furthermore the announced Main Theorem of Kutluhan, Lee, and Taubes \cite{hfishm} shows $\hm(Y,\mathfrak{s},\mathbb{Z}) = HF^\infty(Y,\mathfrak{s},\mathbb{Z})$.  Thus, Theorem~\ref{maintheorem} is already known with $\mathbb{Q}$-coefficients.
To illustrate the importance of coefficients, we point out that the three-manifolds $M_{2n-1}$ and $M_{2n}$ defined in Example 3.3 of \cite{surgeryequivalence} have isomorphic $\hcinfty(\cdot,\mathbb{Q})$ for $n \geq 1$, but different $\hcinfty(\cdot,\mathbb{F})$ for their unique torsion Spin$^c$ structures \cite{hfcoho}.  In fact, this observation combined with the fact that $\hcinfty(M_{2n-1},\mathbb{F})$ and $\hcinfty(M_{2n-1},\mathbb{Q})$ have equal rank proves that there is always 2-torsion in $HF^\infty(M_{2n},\mathbb{Z})$ when $n \geq 1$.
\end{remark}

In Section 2 of \cite{hflz} it is shown that $\hfboldinfty$ vanishes for non-torsion Spin$^c$ structures.  Since the link surgery formula is proved for $\hfboldinfty$ instead of $HF^\infty$, the methods of this paper cannot be used to calculate $HF^\infty$ for non-torsion Spin$^c$ structures.  However, it is pointed out in the same section that for torsion Spin$^c$ structures, $HF^\infty$ is completely determined by $\hfboldinfty$; in other words, $HF^\infty_i \cong \mathbf{HF}^\infty_i$.  Therefore, for the purposes of this paper, we are content to work with $\hfboldinfty$ when studying torsion $\mathfrak{s}$.  From now on, the coefficients of $\cinfty$ will also be $\fpowers$.

Instead of working with the universal coefficients spectral sequence, a different approach is taken to calculate $\hfboldinfty$.  We will express $\hfboldinfty(Y,\mathfrak{s}_0)$ as the homology of a hypercube of chain complexes by \cite{hflz}; this will be the structure on which we construct a different spectral sequence, this time coming from a filtration more reminiscent of the quantum grading in Khovanov homology.

Given a framed link $(L,\Lambda)$ in $S^3$, Manolescu and Ozsv\'ath construct the link surgery formula: an infinite product of hypercubes of chain complexes over $\fpowers$ whose total complex has homology isomorphic to $\hfboldinfty(S^3_\Lambda(L))$ \cite{hflz}.  This complex is built up from a generalization of the mapping cone construction for integral surgeries of \cite{hfkz}, but applied for each sublink of $L$.

By the work of Section 2 in \cite{hfcoho}, we only need to calculate $\hfboldinfty$ in the case that $Y$ is 0-surgery on certain links with all pairwise linking zero; we call such a link \emph{homologically split}.  For 0-surgery on a homologically split link with $\ell$ components, the associated hypercubes of chain complexes for a fixed Spin$^c$ structure naturally take the shape of a single hypercube, $\{0,1\}^\ell$, in the sense that there is only a single Heegaard Floer complex associated to each vertex; therefore, the complex is finite dimensional over $\fpowers$.  We place a filtration on the complex based entirely on the location in the cube.  This will induce the spectral sequence we would like to study.

For the unique torsion Spin$^c$ structure on such a $Y$, denote the complex for 0-surgery on $L$ coming from the surgery formula by $\C^\infty(L,\Lambda_0,\mathbf{0})$ (see Section 7 of \cite{hflz}), where $\Lambda_0$ represents the 0-framing.  We establish the following, analogous to Conjecture 4.10 of \cite{plumbed}.

\begin{theorem} \label{spectralcalc} Consider a homologically split link $L \subset S^3$.  There is a filtration on $\C^\infty(L,\Lambda_0,\mathbf{0})$ such that the induced spectral sequence has the following properties.  The first two pages have vanishing differentials and $E_3 \cong \Lambda^*(H^1(Y;\mathbb{Z})) \otimes \mathbb{F}[[U,U^{-1}]$.  Furthermore, via this identification $d_3:\Lambda^i(H^1(Y;\mathbb{Z}) \otimes \mathbb{F}) \otimes U^j \rightarrow \Lambda^{i-3}(H^1(Y;\mathbb{Z}) \otimes \mathbb{F}) \otimes U^{j-1}$ is given by

\begin{equation} \label{differential}
\phi^{k_1} \wedge \ldots \wedge \phi^{k_i} \mapsto \iota_{\mu_Y} (\phi^{k_1} \wedge \ldots \wedge \phi^{k_i}),
\end{equation}
where $\mu_Y$ is the integral triple cup product form, $\mu_Y(\phi^{k_1} \wedge \phi^{k_2} \wedge \phi^{k_3}) = \langle \phi^{k_1} \smile \phi^{k_2} \smile \phi^{k_3},[Y] \rangle$, thought of as a 3-form on $H^1(Y;\mathbb{Z})$.  Finally, the higher differentials vanish.

\end{theorem}

The reader familiar with cup homology will immediately recognize the complex $(E_3,d_3)$.

\begin{definition}(Definition 8 of \cite{thommark})
The relatively-graded chain complex $C^\infty_*(S^3_{\mathbf{0}}(L))$ is exactly $(E_3,d_3)$ under the identifications of Theorem~\ref{spectralcalc}.
\end{definition}

The first three differentials in Theorem~\ref{spectralcalc} will be calculated by applying our understanding of $b_1(Y) \leq 3$ from Theorem 10.1 of \cite{hfpa} and Theorem 1.3 of \cite{hfcoho}.  It will be easy to show that this gives upper bounds on the rank of $HF^\infty$ for arbitrary $Y$.  The rest of the paper is then devoted to proving inductively that the higher differentials in the above spectral sequence vanish. The idea is to choose a component $K$ in $L$ and replace it by $K_1 \# K_2$, where 0-surgery on each $(L - K) \cup K_i$ yields a simpler cup product structure.  We will study the link surgery formula for $L$ as a combination of the complexes with each $K_i$ in place of $K$. This will enable us to show that the higher differentials vanish based on knowledge of their vanishing for the simpler links.  From this, we can easily deduce Theorem~\ref{maintheorem} for all three-manifolds.


\begin{remark}
It seems likely that if the link surgery formula could be proven over $\mathbb{Z}$, then the arguments of this paper could be used to prove Theorem~\ref{maintheorem} for $\mathbb{Z}$-coefficients as well. 
\end{remark}

\section*{Acknowledgments} I would like to thank Ciprian Manolescu for his guidance on the link surgery formula and general advice on this problem.

\section{Review of the Link Surgery Formula} \label{surgeryreview}
We assume familiarity with Heegaard Floer homology of three-manifolds and links, as in \cite{hfinvariance}, \cite{hfl}.  We now give a brief overview of the link surgery formula of Manolescu and Ozsv\'ath \cite{hflz}.  

Their machine takes as input a framed link $(L,\Lambda)$ in $S^3$ and outputs a special type of chain complex with homology isomorphic to $\hfboldinfty(Y_\Lambda(L))$.  Recall that $\hfboldinfty$ comes from the complex $\mathbf{CF}^\infty = CF^\infty \otimes \fpowers$.  While we only work with $\hfboldinfty$ in this paper, their surgery formula is done for all flavors of Heegaard Floer homology.

In order to explain the link surgery formula, there is a very large amount of notation and formalism required simply to state the theorem.  Therefore, we will first give a complete description in the case that $L$ is a knot to give a more concrete set-up.  Then we will give the general framework, but with slightly less details.  Everything here will be presented for $\hfboldinfty$, but the same framework applies to the other flavors as well.  Finally, for notation, we use $x \vee y$ to denote $\max\{x,y\}$.

\subsection{Surgery Formula for Knots}
We begin with an oriented knot, $K \subset S^3$.  We will restate (without proof) a well-known formula for $\hfboldinfty(S^3_n(K))$ (compare with Theorem 1.1 of \cite{hfkz}).
Begin with a doubly-pointed Heegaard diagram for $K$, $\hyperbox^K = (\Sigma,\boldalpha,\boldbeta,z,w)$, and a Heegaard diagram for $S^3$, $\hyperbox^\emptyset = (\Sigma,\boldalpha',\boldbeta',w')$, so they have the same underlying surface.  Let's suppose for simplicity that after removing the basepoint $z$, now a Heegaard diagram for $S^3$, that we can relate $\hyperbox^K$ to $\hyperbox^\emptyset$ by a sequence of handleslides and isotopies (this means the isotopies of curves do not cross the basepoint and there are no (de)stabilizations).  Construct a finite sequence of Heegaard diagrams, $\hyperbox^{K,+K}$, which begins at $\hyperbox^K$ with $z$ removed and follows the sequence of isotopies and handleslides, terminating at $\hyperbox^\emptyset$.  We define $\hyperbox^{K,-K}$ analogously for the removal of $w$.

First, define $\mathbb{H}(K) = \mathbb{Z}$ and $\overline{\mathbb{H}}(K) = \mathbb{H}(K) \cup \{-\infty,+\infty\}$; also, let  $\mathbb{H}(\emptyset) = 0$ and $\overline{\mathbb{H}}(\emptyset) = 0 \cup \{-\infty,+\infty\}$.  Note that there are two sublinks of $K$, namely $K$ and $\emptyset$.  If we use $K$ or $+K$, we will mean that it has the induced orientation; $-K$ will refer to the reversed orientation.  Fix $s \in \overline{\mathbb{H}}(K)$ and an oriented sublink, $\vec{M} \subset K$.  Define

\[
p^{\vec{M}}(s) = \left\{
\begin{array} {rl}
+\infty &\text{ if } \vec{M} = +K, \\
-\infty & \text{ if } \vec{M} = -K, \\
s & \text{ if } M = \emptyset.
\end{array}
\right .
\]
Similarly, define $\psi^{\vec{M}}(s) = +\infty$ if $M = K$ and set $\psi^{\emptyset}(s) = s$.

We want to construct two complexes for each $s \in \mathbb{H}(K)$, namely one for $\hyperbox^\emptyset$ and one for $\hyperbox^K$.  The first complex is given by $\mathfrak{A}^\infty(\hyperbox^\emptyset,p^K(s)) = \mathfrak{A}^\infty(\hyperbox^\emptyset,+\infty) = \mathbf{CF}^\infty(\hyperbox^\emptyset)$.
The complex for $\hyperbox^K$ will be more complicated and will actually depend on $s$.
Recall that there is an absolute Alexander grading on $\mathbb{T}_\alpha \cap \mathbb{T}_\beta$ coming from $\hyperbox^K$ which satisfies a homological symmetry about 0 and $A(x) - A(y) = n_z(\phi) - n_w(\phi)$ for any $\phi \in \pi_2(x,y)$.  With this, we define the complex $\mathfrak{A}^\infty(\hyperbox^K,s)$ to have the same chain groups as $\mathbf{CF}^\infty(\hyperbox^K)$, the free module over $\fpowers$ generated by $\mathbb{T}_\alpha \cap \mathbb{T}_\beta$.  It has the differential
\[
\partial(x) = \destab^\emptyset(x) =\sum_{y \in \mathbb{T}_\alpha \cap \mathbb{T}_\beta} \sum_{\phi \in \pi_2(x,y), \mu(\phi)=1} \#(\mathcal{M}(\phi)/\mathbb{R}) \cdot U^{E_s(\phi)} y,
\]
where $E_s(\phi) = (A(x)-s) \vee 0 - (A(y)-s) \vee 0 + n_w(\phi)$.  Notice that as $s$ becomes very positive (respectively negative), $\partial$ is only counting $w$ (respectively $z$).

We would like a way to relate these complexes.  Define the \emph{inclusions}, $\proj^{\pm K}_s : \mathfrak{A}^\infty(\hyperbox^K,s) \rightarrow \mathfrak{A}^\infty(\hyperbox^K,p^{\pm K}(s))$ by
\[
\proj^{\pm K}_{s}(x) = U^{(\pm(A(x) - s)) \vee 0 } x.
\]
This essentially corresponds to removing $z$ or $w$ from the Heegaard diagram, since $s$ is sent to $\pm\infty$.  We set $\proj^{\emptyset}_s$ to simply be the identity.  Recall that we had sequences of Heegaard diagrams, $\hyperbox^{K,\pm K}$, relating $\hyperbox^K$ with $z$ or $w$ removed to $\hyperbox^\emptyset$.  Each isotopy or handleslide induces a chain map between Floer complexes by counting triangles \cite{hfinvariance}.  Composing these induced maps along the sequence results in the \emph{destabilization} maps $\destab^{\pm K}_{p^{\pm K}(s)}: \mathfrak{A}^\infty(\hyperbox^K,p^{\pm K}(s)) \rightarrow \mathfrak{A}^\infty(\hyperbox^\emptyset,\psi^{\pm K}(p^{\pm K}(s)))$.  The composition $\destab^{\vec{M}}_{p^{\vec{M}}(s)} \circ \proj^{\vec{M}}_s$ is denoted $\Phi^{\vec{M}}_s$, and thus $\Phi^{\emptyset}_s(x) = \partial(x)$.

Consider the following complex composed of all the smaller complexes we have built up:
\[
\C^\infty(\hyperbox,n) = \prod_{s \in \mathbb{H}(K)} (\mathfrak{A}^\infty(\hyperbox^K,s) \oplus \mathfrak{A}^\infty(\hyperbox^{\emptyset},\psi^K(s)))
\]
with differential given by
\[
\destab^\infty (s,x) = (s + n, \Phi^{-K}_{\psi^M(s)}(x)) + (s,\Phi^{+K}_{\psi^M(s)}(x)) + (s,\Phi^\emptyset_{\psi^M(s)}(x)),
\]
for $x \in \mathfrak{A}^\infty(\hyperbox^{K-M},\psi^M(s))$.  The $s$ in the first component is simply serving as an index.  Here we are using the convention that $\Phi^{\pm K}_s(x) = 0$ if $x \in \mathfrak{A}^\infty(\hyperbox^{\emptyset},\psi^K(s))$ (in other words, if $x$ is not in the domain).
We now have a slightly altered version of the integer surgery formula for knots (see Theorem 7.5 of \cite{hflz})

\begin{theorem}(Ozsv\'ath-Szab\'o)
The homology of the complex $\C^\infty(\hyperbox,n)$ is isomorphic to $\hfboldinfty(S^3_n(K))$.
\end{theorem}

\subsection{Spin$^c$ Structures}
We now generalize the construction above to arbitrary framed links.  For simplicity, we will assume that Heegaard diagrams for links have exactly one $z$ basepoint for each component, but may (and will) have additional $w$ basepoints in the diagram not on any component of the link.  Also, we will ignore all details about admissibility of the Heegaard diagrams as well (see Section 4 of \cite{hflz}).

The starting point will be an oriented link $\vec{L}$ in $S^3$ with components $K_1,\ldots,K_n$ and a framing $\Lambda$ telling us how to perform surgery on $L$.  The framing $\Lambda$ will be given as the linking matrix of the resulting 3-manifold after surgery; diagonal entries are the surgery coefficients and the off-diagonal entries are the pairwise linking numbers of the components.  Note that we may think of the row-vectors $\Lambda_i$ as elements in $H_1(S^3-L)$.  When we are considering oriented sublinks, $\vec{M}$ will refer to an arbitrary orientation, while $M$ with no vector decoration will indicate that $M$ has the orientation induced from $L$.

As in other Heegaard Floer homologies, we want to see where the Spin$^c$ structures appear in our theory.  It will be necessary to also relate the relative Spin$^c$ structures defined on $S^3 - L$ to those on $S^3 - M$ for sublinks $M \subset L$.

Define the affine space $\mathbb{H}(L)= \bigoplus^n_{i=1} \mathbb{H}(L)_i$, where
\[
\mathbb{H}(L)_i = \frac{lk(K_i,L-K_i)}{2} + \mathbb{Z}.
\]

It is not hard to see that as lattices Spin$^c(S^3_{\Lambda}(L)) \cong \mathbb{H}(L)/\Lambda$ (where $/\Lambda$ means quotienting out by the action of each row-vector of $\Lambda$, $\Lambda_i$, on the lattice); it turns out that such an identification can be made explicitly.  Therefore, we will often refer to Spin$^c$ structures as equivalence classes $[\mathbf{s}]$.  We extend these lattices to $\overline{\mathbb{H}}(L)_i = \mathbb{H}(L)_i \cup \{+\infty,-\infty\}$ and $\overline{\mathbb{H}}(L) = \oplus^n_{i=1} \overline{\mathbb{H}}(L)_i$.

Let $I_+(\vec{L},\vec{M})$ be the set of indices of components of $M$ which are consistently oriented with $L$.  The remaining components of $M$ form $I_-(\vec{L},\vec{M})$.  We define the maps $p^{\vec{M}}_i:\mathbb{H}(L)_i \rightarrow \mathbb{H}(L)_i$ by
\[
p^{\vec{M}}_i(s_i) = \left\{
\begin{array}{rl}
+\infty & \text{ if } i \in I_+(\vec{L},\vec{M}), \\
-\infty & \text{ if } i \in I_-(\vec{L},\vec{M}), \\
s_i & \text{ otherwise.}
\end{array}
\right.
\]

We can then apply restrictions $p^{\vec{M}}(\mathbf{s}) = (p^{\vec{M}}_1(s_1),\ldots,p^{\vec{M}}_n(s_n))$.  This will allow us to remove the components of $M$, but still keep track of Spin$^c$ structures consistently.

By viewing $\mathbb{H}(L)$ as an affine space over $H_1(S^3 - L)$ we can define the map $\psi^{\vec{M}}: \mathbb{H}(L) \rightarrow \mathbb{H}(L-M)$ by $\psi^{\vec{M}}(\mathbf{s}) = \mathbf{s} - [\vec{M}]/2$.  In other words, we ignore the components of $\mathbf{s}$ coming from $\vec{M}$, but we must change the remaining components based on their linking with the components of $M$.  We extend this to go from $\overline{\mathbb{H}}(L)$ to $\overline{\mathbb{H}}(L-M)$ in the obvious way.

With this we can define a new Heegaard Floer complex for each choice of $\mathbf{s}$.  Begin with a Heegaard diagram for $L$, $\hyperbox^L$.  Recall that there is an Alexander grading for each component of $L$ on $\mathbb{T}_\alpha \cap \mathbb{T}_\beta$, again  given by making absolute the relative grading $A_i(x) - A_i(y) = n_{z_i}(\phi) - n_{w_i}(\phi)$ (we require the Alexander grading of link Floer homology to be symmetric about 0).

For each $\mathbf{s_0} \in \overline{\mathbb{H}}(L)$ and each $M \subset L$, we will define the complex $\mathfrak{A}^\infty(\hyperbox^{L-M},\psi^M(\mathbf{s_0}))$.  For notation, set $\mathbf{s} = \psi^M(\mathbf{s_0})$.  The chain groups will all be the same, freely generated over $\mathbb{F}[[U_1,\ldots,U_n,U_1^{-1},\ldots,U_n^{-1}]$ by $\mathbb{T}_\alpha \cap \mathbb{T}_\beta$; here $n$ is the number of $w$ basepoints.  The differential will be defined by $D^0 = \partial : \mathfrak{A}^\infty(\hyperbox^{L-M},\psi^M(\mathbf{s_0})) \rightarrow \mathfrak{A}^\infty(\hyperbox^{L-M},\psi^M(\mathbf{s_0}))$, which is given by
\begin{equation*}
\partial(x) = \sum_{y \in \mathbb{T}_\alpha \cap \mathbb{T}_\beta} \sum_{\phi \in \pi_2(x,y), \mu(\phi)=1} \#(\mathcal{M}(\phi)/\mathbb{R}) \cdot U_1^{E^1_{s_1}(\phi)} \ldots U_n^{E^n_{s_n}(\phi)} y,
\end{equation*}
where
\begin{equation*}
E^i_{s_i}(\phi)=(A_i(x)-s_i) \vee 0 - (A_i(y)-s) \vee 0 + n_{w_i}(\phi) .
\end{equation*}
If $s_i$ is very positive (or negative), then these counts are again just $n_{w_i}(\phi)$ (or $n_{z_i}(\phi)$); also, we must use the convention that $\infty - \infty = 0$ so this is consistent when $s_i= -\infty$.  Therefore, setting some $s_i$ to $+\infty$ is the same thing as forgetting the $i$th component of the link and having an additional basepoint $w_i$.

\subsection{Complete Systems of Hyperboxes}
\begin{definition} An $n$\emph{-dimensional hyperbox of size }$\mathbf{d}=(d_1,\ldots,d_n) \in \mathbb{N}^n$ is the following subset of $\mathbb{N}^n$

\[
\mathbb{E}(\mathbf{d}) = \{\varepsilon = (\varepsilon_1,\ldots,\varepsilon_n)| 0 \leq \varepsilon_i \leq d_i \}
\]
If $\mathbf{d} = (1,\ldots,1)$, then $\mathbb{E}(\mathbf{d})$ is a \emph{hypercube}.  The \emph{length} of $\varepsilon$, $\|\varepsilon\|$, is given by $\sum_i \varepsilon_i$.  The elements of $\mathbb{E}(\mathbf{d})$ are called \emph{vertices}.
\end{definition}

Also, there is a natural partial order on $\mathbb{E}(d)$ given by $\varepsilon \leq \varepsilon'$ if and only if $\varepsilon_i \leq \varepsilon_i'$ for all $i$.  Two vertices in the hyperbox are \emph{neighbors} if they differ by an element of $\{0,1\}^n$.  The important example to keep in mind is given by the $n$-dimensional hypercube determined by the set of sublinks of the $n$-component link $L$.  We identify the sublinks of $L$ with the vertices of $\{0,1\}^n$ by setting $\varepsilon(M)_i$ to be 1 if $K_i \subset M$ and 0 otherwise.

\begin{definition}
An $n$-\emph{dimensional hyperbox of chain complexes of size }$\mathbf{d}$ is a collection of chain complexes $(\C_*^\varepsilon,D^0_\varepsilon)$ for $\varepsilon \in \mathbb{E}(\mathbf{d})$ equipped with additional operators $D^{\varepsilon'}_{\varepsilon}:\C_*^\varepsilon \rightarrow \C_{*+\|\varepsilon'\|-1}^{\varepsilon+\varepsilon'}$, for $\varepsilon' \neq 0$ in $\{0,1\}^n$; the operators are 0 if $\varepsilon+\varepsilon'$ is no longer in the hyperbox.  These operators are required to satisfy the following relation for all $\varepsilon' \in \{0,1\}^n$:

\[
\sum_{\gamma \leq \varepsilon'} D_{\varepsilon+\gamma}^{\varepsilon'-\gamma} \circ D_\varepsilon^\gamma = 0
\]
\end{definition}

The way to think of this is that the $D^{\varepsilon'}$ are chain maps when $\|\varepsilon'\| = 1$ and chain homotopies for $\|\varepsilon'\| = 2$.  The higher maps are chain homotopies of chain homotopies, etc.

This can be made into a total complex as
\[
(\C_* = \bigoplus_\varepsilon C^\varepsilon_{*+\|\varepsilon\|}, D = \sum_{\varepsilon,\varepsilon'} D_{\varepsilon'}^\varepsilon)
\]

We will omit the subscript notation from the $D$ from now on, where it will just be assumed that the map is 0 if the relevant domains don't match up.  Furthermore, $\partial$ will denote $D^0$ at any vertex of the hyperbox.

Let $(\Sigma,\boldalpha,\mathbf{z},\mathbf{w})$ be a Heegaard diagram for a handlebody, with basepoints $\mathbf{z} = \{z_1,\ldots, z_n\}$ and $\mathbf{w} = \{w_1,\ldots,w_n\}$ on $\Sigma - \boldalpha$.

We will assume that all bipartition functions send everything to $\boldbeta$, so we will not worry about defining $\alpha$-hyperboxes or keeping track of bipartition functions.  We will ultimately work with a basic system, so this assumption will not be a problem (see Section 6.7 of \cite{hflz}).

\begin{definition}
An \emph{empty $\beta$-hyperbox of size} $\mathbf{d}$, $\hyperbox$, is a collection of isotopic sets of $\boldbeta$-curves on $\Sigma-\mathbf{z}-\mathbf{w}$, $\{\boldbeta^\varepsilon\}_{\varepsilon \in \mathbb{E}(\mathbf{d})}$.  A \emph{filling} of $\hyperbox$ is a choice of elements $\Theta_{\varepsilon,\varepsilon'} \in \mathfrak{A}^\infty(\mathbb{T}_{\beta_\varepsilon},\mathbb{T}_{\beta_{\varepsilon'}},\mathbf{0})$ for any neighbors $\varepsilon < \varepsilon'$.  These are required to satisfy equation (50) in \cite{hflz}, namely summing over the polygon maps associated to each possible sequence $\Theta_{\varepsilon_1,\varepsilon_2}, \Theta_{\varepsilon_2,\varepsilon_3}, \ldots \Theta_{\varepsilon_{l - 1},\varepsilon_l}$ in the Heegaard multiple $(\Sigma,\boldalpha,\boldbeta_{\varepsilon_1},\ldots,\boldbeta_{\varepsilon_l})$, is identically 0.  If $\|\varepsilon - \varepsilon'\| = 1$, $\Theta_{\varepsilon,\varepsilon'}$ must also correspond to a cycle generating the top-dimensional homology group of $\widehat{\mathfrak{A}}(\mathbb{T}_{\beta_\varepsilon},\mathbb{T}_{\beta_{\varepsilon'}},\mathbf{0})$ (where $\widehat{\mathfrak{A}}$ is given by setting one $U_i$ to 0).    
\end{definition}

Therefore, for us a \emph{hyperbox of Heegaard diagrams for $L$} is simply a set of $\boldalpha$-curves and an empty $\boldbeta$-hyperbox equipped with a choice of filling such that each $(\Sigma,\boldalpha,\boldbeta_\varepsilon,\mathbf{z},\mathbf{w})$ is a Heegaard diagram for $L$.

\begin{remark} Given a fixed $\mathbf{s} \in \mathbb{H}(L)$, we can create a hyperbox of chain complexes from a hyperbox of Heegaard diagrams as follows: for each $\varepsilon \in \mathbb{E}(\mathbf{d})$ we set $(\C^{\varepsilon(M)}_{\mathbf{s}},D^0)$ to be $\mathfrak{A}^\infty(\mathbb{T}_\alpha,\mathbb{T}_{\beta_{\varepsilon(M)}},\psi^{M}(\mathbf{s}))$.  If $\|\varepsilon' - \varepsilon \| = 1$, then the chain map $D^{\varepsilon' - \varepsilon}_{\varepsilon}$ consists of counting triangles in the Heegaard triple $(\boldalpha,\boldbeta_\varepsilon,\boldbeta_{\varepsilon'})$ with fixed generator $\Theta_{\varepsilon,\varepsilon'}$.  The higher homotopies are defined similarly; we sum up the corresponding holomorphic polygon counts over a specified sequence of the $\Theta$ elements in the Heegaard multiple $(\Sigma,\boldalpha,\boldbeta_{\varepsilon},\ldots,\boldbeta_{\varepsilon'})$.
\end{remark}

It is a lemma of Manolescu and Ozsv\'ath that any empty $\boldbeta$-hyperbox admits a filling and thus every empty $\boldbeta$-hyperbox can be made into a hyperbox of Heegaard diagrams.

Given an $m$-component sublink $\vec{M} \subset L' \subset L$ and a hyperbox of Heegaard diagrams $\hyperbox$ for $L'$, we construct a new hyperbox $r_{\vec{M}}(\hyperbox)$.  This is defined as follows.  Remove the $z_i$ on components of $I_+(\vec{L},\vec{M})$ from each Heegaard diagram in $\hyperbox$; remove the $w_i$ that correspond to components of $I_-(\vec{L},\vec{M})$ and relabel them as $z_i$.  Note that this is now a hyperbox for $L'-M$.

\begin{definition} A \emph{hyperbox for the pair} $(\vec{L'},\vec{M})$, $\hyperbox^{\vec{L'},\vec{M}}$ is an $m$-dimensional hyperbox of Heegaard diagrams for $\vec{L'}-M$.
\end{definition}

Let's study some special cases.  If $M = \emptyset$, then a hyperbox for the pair $(\vec{L'},\emptyset)$ is a single Heegaard diagram, which we denote by $\hyperbox^{\vec{L'}}$.
If $M$ is a single component $K$, then we have $\hyperbox^{\vec{L'},\pm K}$ is a one-dimensional hyperbox, or in other words, a finite sequence of Heegaard diagrams.  For the integer surgeries formula, this related $\hyperbox^K$ with $z$ or $w$ removed to $\hyperbox^\emptyset$; this is exactly the idea that we would like to keep in mind.  This box is going to tell us how to define the maps analogous to $\destab^{\pm K}$.

Given a sublink, $M' \subset M$, there is a hyperbox for $(\vec{L}-M',\vec{M}-M')$ inside of the size $\mathbf{d}$ $\hyperbox^{\vec{L},\vec{M}}$.  This hyperbox, $\hyperbox^{\vec{L},\vec{M}}(M',M)$, is given by the sub-hyperbox with specified corners $\mathbf{d} \cdot \varepsilon(M')$ and $\mathbf{d} \cdot \varepsilon(M)$ (here we are doing componentwise multiplication).

For knots, we simply pointed out that for large $|s|$, $\mathfrak{A}^\infty(\hyperbox^K,s)$ behaves as though there is either no $z$ or no $w$ basepoint and can be compared to $\mathfrak{A}^\infty(\hyperbox^\emptyset,\psi^K(s))$.  We now state the analogous requirement for comparing hyperboxes with certain basepoints removed.
Say that two hyperboxes are \emph{compatible} if $\hyperbox^{\vec{L},\vec{M}}(\emptyset,M') \cong r_{\vec{M}-M'}(\hyperbox^{\vec{L},\vec{M'}})$ for $M'$ a sublink consistently oriented with $\vec{M} \subset L$.  Similarly, $\hyperbox^{\vec{L},\vec{M}}$ and $\hyperbox^{\vec{L}-M',\vec{M}-M'}$ are compatible if $\hyperbox^{\vec{L},\vec{M}}(M',M) \cong \hyperbox^{\vec{L}-M',\vec{M}-M'}$.  Here, the relation `$\cong$' means that the hyperboxes of Heegaard diagrams are related by a single isotopy.
In other words, there is a single isotopy of $\Sigma$ not passing any curves over basepoints, independent of $\varepsilon$, which takes the Heegaard diagrams at vertex $\varepsilon$ on one hyperbox to the Heegaard diagram at vertex $\varepsilon$ on the other.

\begin{definition}
A \emph{complete system of hyperboxes for $L$ } is a collection of hyperboxes for each pair $(\vec{L'},\vec{M})$, $\hyperbox^{\vec{L'},\vec{M}}$, such that for any sublink of $M'$ with orientation induced by $\vec{M}$, $\hyperbox^{\vec{L'},\vec{M}}$ is compatible with $\hyperbox^{\vec{L'} - M',\vec{M}-M'}$ and $\hyperbox^{\vec{L'},\vec{M'}}$.
\end{definition}
Manolescu and Ozsv\'ath construct complete systems of hyperboxes for any oriented link in $S^3$.

\begin{remark} There is an additional technical condition that must be satisfied to be a complete system in the sense of Manolescu and Ozsv\'ath (Definitions 6.25 and 6.26 in \cite{hflz}): it essentially says that the paths traced out by the basepoints on the Heegaard surfaces while passing between the different isotopies of diagrams in the hyperboxes must be nullhomotopic.  This will not be a problem with the special types of complete systems we will work with, so we do not mention this anymore.
\end{remark}

\subsection{The Surgery Complex}
Given a complete system of hyperboxes of Heegaard diagrams for an $n$-component link $L$, $\hyperbox$, we would like to turn
\[
\C^\infty(\hyperbox,\Lambda) = \prod_{\mathbf{s} \in \mathbb{H}(L)} \sum_{M \subset L} \mathfrak{A}^\infty(\hyperbox^{L-M},\psi^M(\mathbf{s}))
\]
into an $n$-dimensional hypercube of chain complexes analogous to the case for knots.  We will set the chain complex at the vertex $\varepsilon(M)$ to be
\[
\C^{\varepsilon(M)} = \prod_{s \in \mathbb{H}(L)} \mathfrak{A}^\infty(\hyperbox^{L-M},\psi^M(\mathbf{s}))
\]
with the differential given by the product of the component-wise differentials.  While these chain complexes do not depend on $\Lambda$, the $D^\varepsilon$ that we will ultimately construct will depend heavily on this choice.

We now want to define the analogue for the maps $\Phi^{\pm K}$ relating the $\mathfrak{A}^\infty$ complexes.  We will construct a map from $\mathfrak{A}^\infty(\hyperbox^{M},\mathbf{s})$ to $\mathfrak{A}^\infty(\hyperbox^{M'},\psi^{M-M'}(\mathbf{s}))$ for each $\vec{M'} \subset M$ and $\mathbf{s} \in \mathbb{H}(M)$.  The first step is to remove the $z$ or $w$ basepoints in $\hyperbox^{M}$ that do not correspond to components of $M'$ to get a Heegaard diagram for $M'$; this corresponds to $\proj$ in the integer surgery formula for knots.

We can define the general inclusions $
\proj^{\vec{M}}_{\mathbf{s}} :\mathfrak{A}^\infty(\hyperbox^{L'},\mathbf{s}) \rightarrow \mathfrak{A}^\infty(\hyperbox^{L'},p^{\vec{M}}(\mathbf{s}))$ by
\[
\proj^{\vec{M}}_{\mathbf{s}}(x) = \prod_{i \in I_+(\vec{L},\vec{M})} U_i^{(A_i(x) - s_i) \vee 0} \prod_{j \in I_-(\vec{L},\vec{M})} U_j^{(s_j - A_j(x)) \vee 0} x.
\]
Note that this is only defined if $s_i$ is not $\pm \infty$ when $i \in I_{\mp}(\vec{L},\vec{M})$; this issue will not arise when we define the total complex.

We would now like to define the destabilizations, $\destab^{\vec{M}}_{p^{\vec{M}}(\mathbf{s})}:\mathfrak{A}^\infty(\hyperbox^{L'},p^{\vec{M}}(\mathbf{s})) \rightarrow \mathfrak{A}^\infty(\hyperbox^{L'-M},\psi^{\vec{M}}(p^{\vec{M}}(\mathbf{s})))$ analogous to the case of $\destab^{\pm K}$.  We first identify $r_{\vec{M}}(\hyperbox^{L'})$ with its corresponding vertex in $\hyperbox^{L',\vec{M}}$, $\hyperbox^{L',\vec{M}}_{(0,\ldots,0)}$, by compatibility.  This induces a map on $\mathfrak{A}^\infty$ by counting triangles, but this must count basepoints with $E_{\mathbf{s}}(\phi)$ instead of $n_w(\phi)$.

For simplicity, we first assume that $\hyperbox^{L',\vec{M}}$ is in fact a hypercube.  The idea is that each way of traversing the edges of the hypercube gives a sequence of isotopies and handleslides from $\hyperbox^{L',\vec{M}}_{(0,\ldots,0)}$ to $\hyperbox^{L',\vec{M}}_{(1,\ldots,1)}$; $\destab^{\vec{M}}$ will measure the failure for the induced triangle maps to commute.  Recall that the filling gives a map from $\mathfrak{A}^\infty(\mathbb{T}_{\alpha},\mathbb{T}_{\beta_{(0,\ldots,0)}},\mathbf{0})$ to $\mathfrak{A}^\infty(\mathbb{T}_{\alpha},\mathbb{T}_{\beta_{(1,\ldots,1)}},\mathbf{0})$ by counts of holomorphic polygons.  We can twist this map by counting basepoints according to the $E^i_s(\phi)$ as opposed to the usual $n_w(\phi)$.  This in fact gives the desired map $\mathfrak{A}^\infty(\hyperbox^{L',\vec{M}}_{(0,\ldots,0)},p^{\vec{M}}({\mathbf{s}})) \rightarrow \mathfrak{A}^\infty(\hyperbox^{L',\vec{M}}_{(1,\ldots,1)},p^{\vec{M}}(\mathbf{s}))$.

If $\hyperbox^{L',\vec{M}}$ is not a hypercube, but instead has size $\mathbf{d}$, then applying the above maps, $\destab^{\vec{M}}$, will go to $\hyperbox^{L',\vec{M}}_{(1,\ldots,1)}$ instead of $\hyperbox^{L'-M}$.  Therefore, we must do what is called \emph{compression} to arrive at $\hyperbox^{L',\vec{M}}_{\mathbf{d}\cdot (1,\ldots,1)} = \hyperbox^{L'-M}$.  If $\vec{M} = \pm K_i$ for a knot $K_i$, then we would like to go from $\hyperbox^{L',\pm K_i}_{0}$ to $\hyperbox^{L',\pm K_i}_{\mathbf{d}}$ (these will be the two vertices in the hypercube).  There exist triangle-counting maps from $\mathfrak{A}^\infty(\hyperbox^{L',\pm K_i}_{j},p^{\pm K}(\mathbf{s}))$ to $\mathfrak{A}^\infty(\hyperbox^{L',\pm K_i}_{j+1},p^{\pm K_i}(\mathbf{s}))$.  In this case we would simply take the map $\tilde{\destab}^{\pm K_i}$ to be the composition of the $d_i$ triangle-counting maps.

In fact, compression will produce a hypercube with vertices given by the corresponding complexes at the far corners of the hyperbox, namely $\tilde{\C}^\varepsilon$ for $\varepsilon \in \{0,1\}^n$ will be given by $\C^{\mathbf{d}\cdot\varepsilon}$.  Also, if $\| \varepsilon - \varepsilon'\| = 1$ and $\varepsilon \geq \varepsilon'$ in the compressed hypercube, then along an edge, $\tilde{\destab}^{\varepsilon - \varepsilon'}$ will simply be given by the composition of the edge maps from the original hyperbox.  However, in general this does not work (a composition of chain homotopies is not a chain homotopy for the compositions).  For illustration, we define the appropriately compressed map for a size $(2,1)$ hyperbox and refer the interested reader to Section 3 in \cite{hflz} for the general case.

\begin{example} Consider a hyperbox of chain complexes, $\C$, of size $(2,1)$.  We can turn this into a hypercube of chain complexes, $\tilde{\C}$ as follows.  Take $\tilde{\C}^{\varepsilon_1,\varepsilon_2} = \C^{2\varepsilon_1,\varepsilon_2}$ and keep $\tilde{D}^0 = D^0$.  In other words, the complexes at the vertices are given by the corners of the hyperbox.  The map $\tilde{D}^{1,0}$ is given by $D^{1,0} \circ D^{1,0}$, while $\tilde{D}^{0,1} = D^{0,1}$.  So far we have not done anything different from above, but $\tilde{D}^{1,1}$ will have to be more complicated.  A standard exercise in homological algebra shows that the correct choice for $\tilde{D}^{1,1}$ is $D^{1,0} \circ D^{1,1}$ + $D^{1,1} \circ D^{1,0}$.  In the Heegaard Floer setting, $D^{1,0}$ and $D^{0,1}$ are triangle-counting maps, while $D^{1,1}$ counts holomorphic rectangles.
\end{example}

Once the correct map from $\mathfrak{A}^\infty(\hyperbox^{L',\vec{M}}_{\mathbf{0}},p^{\vec{M}}({\mathbf{s}}))$ to $\mathfrak{A}^\infty(\hyperbox^{L',\vec{M}}_{\mathbf{d}\cdot \varepsilon(M)},p^{\vec{M}}(\mathbf{s}))$ is defined, we simply apply our identification of this final Heegaard diagram with $\hyperbox^{L'-M}$ to get one last triangle counting map, again by compatibility.  This final composition is the destabilization $\destab^{\vec{M}}_{p^{\vec{M}}(\mathbf{s})}$.

For an arbitrary sublink, $\vec{M}$, define the map, $\Phi^{\vec{M}}_{\mathbf{s}} = \destab^{\vec{M}}_{p^{\vec{M}}(\mathbf{s})} \circ \proj^{\vec{M}}_{\mathbf{s}}$.
The differential $\destab^\infty$ on $\C^\infty(\hyperbox,\Lambda)$ is given by
\[
(\mathbf{s},x) \mapsto \sum_{\vec{N} \subset L-M} (\mathbf{s} + \Lambda_{\vec{L},\vec{N}},\Phi^{\vec{N}}_{\psi^{M}(\mathbf{s})}(x)).
\]
Here, $x \in \mathfrak{A}^\infty(\hyperbox^{L-M},\psi^M(s))$ and $\Lambda_{\vec{L},\vec{N}} = \sum_{i \in I_-(\vec{L},\vec{N})} \Lambda_i$. Note that the sum is over all possible \emph{oriented} sublinks of $L-M$.

Manolescu and Ozsv\'ath prove that this is indeed the total complex of a hypercube of chain complexes.
It is important to note that by construction, if $\mathbf{s} \in [\mathbf{s_0}]$, then $\mathbf{s} + \Lambda_{\vec{L},\vec{N}} \in [\mathbf{s_0}]$ for any $\vec{N}$.  Therefore, $\C^\infty(\hyperbox,\Lambda)$ splits as a sum of complexes corresponding to the Spin$^c$ structures, $\C^\infty(\hyperbox,\Lambda,[\mathbf{s}])$.

We are finally ready to state the link surgery formula.

\begin{theorem} (Manolescu-Ozsv\'ath, Theorem 1.1 of \cite{hflz}) Consider a complete system of hyperboxes,$\hyperbox$, for $\vec{L} \subset S^3$ and a framing $\Lambda$.  Given a Spin$^c$ structure $\mathfrak{s}$ on $S^3_{\Lambda}(L)$ corresponding to $[\mathbf{s}] \in \mathbb{H}(L)/\Lambda$, there is a relatively-$\mathbb{Z}/N\mathbb{Z}$-graded $\mathbb{F}[[U,U^{-1}]$-vector space isomorphism
\[
\hfboldinfty_*(S^3_{\Lambda}(L),\mathfrak{s}) \cong H_*(\C^\infty(\hyperbox,\Lambda,[\mathbf{s}]),\destab^\infty),
\]
where $N$ is the usual divisibility of the Spin$^c$ structure (see, for example, \cite{hfinvariance}).
\end{theorem}

\begin{remark}
While the $\mathfrak{A}^\infty$ complexes are defined over a ring with many formal variables $U_i$, the theorem implies that they become equal in homology.
\end{remark}

\begin{remark}
In the case of a torsion Spin$^c$ structure, this is a relative $\mathbb{Z}$-grading.  We will use the fact that the differential lowers this relative grading by 1; namely we know the exact grading change of any map $\Phi^{\vec{M}}$.  This grading will be the key ingredient for establishing the identification with the complex for cup homology, which is why this argument only works for torsion Spin$^c$ structures.
\end{remark}

\section{Why is the surgery formula special for $\hfboldinfty$?} \label{hfisspecial}

Let's study some special properties of the link surgery formula which are unique to the infinity flavor.  The reason why these properties will not hold for the other flavors is that the inclusion maps $\proj$ will simply not be quasi-isomorphisms, as multiplication by $U$ is not an isomorphism for $\mathbf{CF}^+$ or $\mathbf{CF}^-$.  However, the inclusions are quasi-isomorphisms for $\hfboldinfty$ and thus have the simplest behavior on this flavor; since the inclusions encode the information coming from the link we are performing surgery on (they describe the induced filtrations on $\mathbf{CF}(S^3)$ similar to $CFL$), $\mathbf{CF}^\infty$ will not retain much information about the choice of individual components.  We now make this notion more precise.

Fix a complete system $\hyperbox$ for the framed link $(L,\Lambda)$.  The following is essentially a combination of Proposition 4.2 in \cite{hfcoho} and Lemma 7.9 in \cite{hflz}.

\begin{proposition} \label{vertexiso}
Consider $\mathbf{s}$ in an equivalence class corresponding to a torsion Spin$^c$ structure.  For any component $K_j$, the maps $\Phi^{\pm K_j}_{\mathbf{s}}:\mathfrak{A}^\infty(\hyperbox^M,\mathbf{s}) \rightarrow \mathfrak{A}^\infty(\hyperbox^{M-K_j},\psi^{\pm K_j}(\mathbf{s}))$ are quasi-isomorphisms that lower the relative grading by 1.
\end{proposition}

 For notation, $\C$ will represent the hypercube of chain complexes $\C^\infty(\hyperbox,\Lambda)$, or possibly the subcomplex corresponding to a torsion Spin$^c$ structure when this will not cause confusion.

\begin{remark}  \label{vertexhomology} Consider the complex $\C^{\varepsilon(M)}_{\mathbf{s}} = \mathfrak{A}^\infty(\hyperbox^{L-M},\mathbf{s})$.  This is quasi-isomorphic to $\mathfrak{A}^\infty(\hyperbox^\emptyset,\mathbf{+\infty})$  by applying $\Phi^{\pm K_j}$ for all components $K_j \subset L-M$.  We can conclude that each $\C^\varepsilon_{\mathbf{s}}$ has homology given by $\fpowers$, so it is generated by a single element.
\end{remark}

However, we can do better than this.  Consider a face, $F$, of any dimension in $\{0,1\}^n$.  Let $L_F$ be the sublink consisting of components such that $\Phi^{+K_i}$ does not vanish in $\C_F$ (in other words, both  $\varepsilon_i = 0$ and $\varepsilon_i = 1$ appear in the face).  Define $\mathbb{H}(L):(L,\Lambda|_{L_F})$ to be the quotient of the lattice $\mathbb{H}(L)$ by the sublattice generated by $\Lambda_i$ where $K_i$ is a component of $L_F$.  Furthermore, let $\mathbb{H}(L,\Lambda|_{L_F})$ denote the set of $\mathbf{s}$ in $\mathbb{H}(L)$ with $[\mathbf{s}] \in \mathbb{H}(L):H(L,\Lambda|_{L_F})$.  Let's construct the following module
\[
\C_F = \prod_{\mathbf{s} \in \mathbb{H}(L,\Lambda|_{L_F})} \sum_{\varepsilon(M) \in F} \mathfrak{A}^\infty(\hyperbox^{L-M},\psi^M(\mathbf{s}))
\]
This is naturally a chain complex, even if it is not a sub- or quotient-complex.  This is because any such face-module is the result of a sequence of subcomplexes of quotient-complexes of subcomplexes etc.  Choose a component $K_j$ that is not in $L_F$ such that $\varepsilon_j = 0$ in $F$.  We can construct a new subface complex of the same dimension, $F_j$, given by changing $\varepsilon_j$ to 1 for each $\varepsilon$ and giving it the inherited chain complex structure (this is now a subcomplex).

\begin{lemma} \label{pageiso}
With the notation as above, $C_F$ and $C_{F_j}$ are quasi-isomorphic.
\end{lemma}
\begin{proof}
We now study the induced map $\mathcal{S}^{+K_j} = \sum_{M \subset L_F} \Phi^{+K_j \cup +M}$ from $\C_F$ to the $\C_{F_j}$, where $K_j$ is not a component of $L_F$ (we must take the product over all $\mathbf{s}$).  This is a chain map by construction.  Consider the filtration on the mapping cone of $\mathcal{S}^{+K_j}$ given by $\F_j(x) = -\sum_{i \neq j} \varepsilon_i$.  The only components that preserve the filtration level will be $\partial$ and $\Phi^{K_j}$.  Since the maps $\Phi^{K_j}$ are quasi-isomorphisms between each of the corresponding $\mathfrak{A}^\infty$ complexes, the associated graded is acyclic.  Therefore, the entire mapping cone is acyclic.  Thus, the two face complexes are quasi-isomorphic.
\end{proof}

\begin{remark} This does not imply that all associated face complexes of the same dimension in $\C^\infty(\hyperbox,\Lambda)$ are quasi-isomorphic.
\end{remark}

This does, however, tell us how to relate face complexes to the complexes corresponding to surgery on certain sublinks.

\begin{lemma} \label{removecomponent} Consider the filtration $\mathcal{F}(x) = -\|\varepsilon\|$ defined on $C_F$.  The face complex $C_F$ is filtered quasi-isomorphic (up to an absolute shift in filtration) to the complex $\C^\infty(\hyperbox',\Lambda|_{L_F})$ corresponding to $S^3_{\Lambda|_{L_F}}(L_F)$, where $\hyperbox'$ is a complete system of hyperboxes of Heegaard diagrams for $L_F$ coming from restricting $\hyperbox$.
\end{lemma}
\begin{proof}
For each $K_j$ in $L-L_F$ with $\varepsilon_j = 0$, simply apply $\mathcal{S}^{+K_j}$ to obtain the desired filtered quasi-isomorphisms to the complex with $\varepsilon_j=1$.  After exhausting all such $K_j$, the resulting subcomplex corresponds to $S^3_{\Lambda|_{L_F}}(L_F)$ after an appropriate restriction of $\hyperbox$ (see Section 11 of \cite{hflz}).
\end{proof}

\begin{remark} This lemma will be used repeatedly throughout this paper with the key observation that these quasi-isomorphisms of the faces preserve the filtration $\F$ (up to some absolute shift which we will ignore).
\end{remark}

\section{The First Two Differentials}
In this section, we set the framework to prove Theorem~\ref{spectralcalc} by inducing the correct filtration and dispensing with the first two differentials in the spectral sequence.

Consider an arbitrary $\ell$-dimensional hypercube of chain complexes, $\C$.  We study the filtration $\F_{\C}$ on the total complex given by $\F_{\C}(x) = \ell -\|\varepsilon\|$ for any $x$ in $\C^\varepsilon$.  This induces a spectral sequence with $E_1$ term the homology of the total complex with respect to only $D^0=\partial$, which converges to the homology of the total complex.  

Now, let $\hyperbox$ be a complete system of Heegaard diagrams for $L= K_1 \cup K_2 \cup \ldots \cup K_\ell$, a homologically split link of $\ell$ components.  We will denote by $Y$ the manifold obtained by 0-surgery on each component of $L$.  The framing is $\Lambda_0$, which is simply the 0-matrix.  We can then induce the filtration $\F_{\C}$ on $\C^\infty(\hyperbox,\Lambda_0,\mathbf{0})$, the subcomplex corresponding to the unique torsion Spin$^c$ structure; we will refer to this as the $\varepsilon$-filtration.
Note that this Spin$^c$ structure is represented by a single element in $\mathbb{H}(L)$.

\begin{remark}
Because of this, the complex $\C^\infty(\hyperbox,\Lambda_0,\mathbf{0})$ has only one $\mathfrak{A}^\infty$ complex at each vertex of $\{0,1\}^\ell$; in fact we can see the $E_1$ term must have rank $2^\ell$, by Remark~\ref{vertexhomology}.  Furthermore, we will suppress the orientations of $Y$ and $L$, as this can be seen to not affect any of the calculations.
\end{remark}

We will use heavily the notion of the \emph{depth} of a filtration; this is the maximal difference in the filtration levels between any two elements.  It is clear that all differentials, $d_k$, in the induced spectral sequence from a filtration vanish for $k$ greater than the depth.  For the $\varepsilon$-filtration coming from an $\ell$-component link, the depth is simply $\ell$.

\begin{proposition} \label{firsttwovanish} The first two differentials, $d_1$ and $d_2$, in the spectral sequence vanish.
\end{proposition}

A key ingredient in the proof will be the following theorem due to Ozsv\'ath and Szab\'o.  Note that this can also be quickly proven using the results from Section~\ref{hfisspecial}.

\begin{theorem} [Theorem 10.1 of \cite{hfpa}] Suppose $b_1(Y) \leq 2$ and $\mathfrak{s}$ is torsion.  Then, $HF^\infty(Y,\mathfrak{s})$ is a free $\mathbb{Z}[U,U^{-1}]$-module of rank $2^{b_1(Y)}$.
\end{theorem}

\begin{proof}[Proof of Proposition~\ref{firsttwovanish}]
We prove this by induction on $\ell$.  $\C$ will now refer to $\C^\infty(\mathcal{H},\Lambda_0,\mathbf{0})$.  Since in this subcomplex there can only be one possible value of $\psi^{M}(\mathbf{s})$ (modulo $\infty$'s) for each sublink, namely $\mathbf{0} \in \mathbb{H}(L-M)$, we will omit this from the notation.

Let us now show that $d_1 \equiv 0$.  As a warm-up, if $\ell=0$, then $\hfboldinfty(Y) \cong \fpowers$.  Since $E_1 \cong E_\infty \cong H_*(\C^0,\partial) \cong \fpowers$, we must have that $d_1=0$. Now, for $\ell = 1$, we see that $\hfboldinfty(Y) \cong \fpowers \oplus \fpowers$ by Theorem 4.2.  Again,
\[E_1 \cong H_*(C^0,\partial) \oplus H_*(C^1,\partial) \cong \fpowers \oplus \fpowers,
\]
so $d_1=0$.
Suppose that $d_1$ vanishes for any link with $\ell$ components.  Let $L$ have $(\ell+1)$-components and have $d^i_1$ represent the component of the differential $d_1$ which maps from $\C^\varepsilon$ to $\C^{\varepsilon+\tau_i}$, where $\tau_i = (0,\ldots,0,1,0,\ldots,0)$ with the 1 in component $i$.  Now, let's consider the subcomplex $\C_j = \bigoplus_{\varepsilon_j=1} \C^\varepsilon$, which corresponds to 0-surgery on the $\ell$-component sublink $L'=L - K_j$, for some $j \neq i$.  Note that the inclusion of $\C_j$ with its own $\varepsilon$-filtration as an $\ell$-component link coming via the identifications of Section~\ref{hfisspecial} into $\C$ is a morphism of filtered complexes.  Therefore, this induces homomorphisms between $E^{\C_j}_n$ and $E^{\C}_n$ for all $n$ (see, for example, \cite{usersguide}).  It is easy to see that this is an injection on the $E_1$ terms.  Thus, we must have that the image of the kernel of the corresponding map $d^i_1$ for the 0-surgery complex for $L'$ is exactly the kernel of $d^i_1|_{\C_j}$ .  By assumption, $d_1$ is identically 0 for the complex associated to an $\ell$-component sublink, so $d^i_1|_{\C_j}=0$.

We now want to see that $d^i_1$ is 0 on the quotient complex $\bigoplus_{\varepsilon_j=0} \C^\varepsilon = \C/\C_j$.  Since $d^i_1$ has no nonzero component from $\C/\C_j$ to $\C_j$, this will show that $d^i_1$ is identically 0 everywhere.  We can in fact identify $\C/\C_j$ with $\C_j$, simply by applying the filtered quasi-isomorphism
\[
\mathcal{S}^{+K_j} = \sum_{M \subset L - K_j} \Phi^{+K_j \cup +M}
\]
from Lemma~\ref{pageiso}.  Therefore, $d^i_1$ is 0 on $\C/\C_j$.  Repeating this argument for various $i$ and $j$, we obtain $d_1 \equiv 0$.

In fact, we can repeat this argument to prove that $d_2$ is identically 0 as well.  For $\ell = 0$ and $1$, this is trivial simply by the depth of the filtration.  Thus, we begin our analysis with $\ell = 2$.  As before, from \cite{hfpa} we have that $\hfboldinfty(Y)$ has rank $4$ over $\fpowers$.  However, we know that the total rank of the $E_1$ page must in fact be $2^\ell = 4$.  Therefore, $d_2$ and the higher differentials vanish.

Now, for the induction step, we want to notice that $E_2 \cong E_1$ by the previous argument that $d_1 = 0$.  Therefore, we have the same injectivity properties on the $E_2$ pages coming from the inclusion of the faces $\C_j$.  We get that $d_2$ is 0 on $\C_j$ again by including the corresponding complex for sublink $L-K_j$, which has vanishing $d_2$ by induction.  For $\C/\C_j$, we have a similar statement to the $d_1$ case, which is that $d^{i_1,i_2}_2 = 0$ for $i_1,i_2 \neq j$, where $d^{i_1,i_2}_2$ is the component of $d_2$ which maps from $\C^\varepsilon$ to $\C^{\varepsilon + \tau_{i_1} + \tau_{i_2}}$.  This follows by again identifying $\C/\C_j$ with $\C_j$ via a filtered quasi-isomorphism.  After doing this for different values of $j$, we see that $d^{i_1,i_2}_2 = 0$ for all pairs $(i_1,i_2)$.  This shows $d_2$ vanishes.
\end{proof}

\section{The Third Differential}
We may now identify the $E_3$ page with $E_1$ in a natural way.  However, this still does not yet look like an exterior algebra.  This next lemma creates the image we desire.  To motivate this, recall that the $E_1$ term has rank $2^\ell$, which is exactly the total rank of the exterior algebra for an $\ell$-dimensional space.  For notational purposes, let $\fexterior(Y) = \Lambda^*(H^1(Y;\mathbb{Z})) \otimes \mathbb{F}$.

In the following lemma, we must take the word `natural' with a grain of salt.  As the link surgery formula only establishes a relative grading on $\C^\infty(\hyperbox,\Lambda_0,\mathbf{0})$, we simply fix a choice of absolute grading on this complex.  For the remainder of the paper, we will keep everything fixed with respect to this choice of absolute grading, in the sense that inclusions of complexes corresponding to sublinks should respect this absolute grading.  In this sense the identifications we will make will be `natural' with respect to these inclusions.

\begin{lemma} \label{d3calc}
Let $L$ be a homologically split link and consider the $\varepsilon$-filtration on $\C^\infty(\hyperbox,\Lambda_0,\mathbf{0})$.  There is a natural association of the $E_3$ page of the induced spectral sequence with $\fexterior(Y) \otimes \fpowers$, such that $d_3: \Lambda^i_{\mathbb{F}}(Y) \otimes U^j \rightarrow \Lambda^{i-3}_{\mathbb{F}}(Y) \otimes U^{j-1}$.
\end{lemma}

\begin{proof}
Because $d_1 = d_2 = 0$, we need only make the identification with the exterior algebra for the $E_1$ term.  Recall from Remark~\ref{vertexhomology} that each vertex of the hypercube has $\partial$-homology isomorphic to $\hfboldinfty(S^3)$ (there is only one $\mathbf{s}$ associated to each $\C^\varepsilon$ by our choice of link and Spin$^c$ structure).  Therefore, the term $E_1^p$, filtration level $p$, of the corresponding spectral sequence will simply be $\binom{\ell}{p}$ copies of $\hfboldinfty(S^3)$.

We can identify the $E_1$ page with $\fexterior \otimes \fpowers$ as follows.  Choose a basis $\{x^i\}$ for $H^1$ such that $x^i$ corresponds to the Hom-dual of the meridian of $K_i$ (this can be done since $H_1$ is torsion free).  Suppose that $\|\varepsilon\| = \ell-k$.  First, identify 1 in $\Lambda^*_{\mathbb{F}}$ with the generator of $H_*(\C^{(1,\ldots,1)},\partial) \cong \hfboldinfty(S^3)$ with fixed absolute grading 0.  We can then choose a generator of $H_*(\C^{\varepsilon(M)},\partial)$ to be the image of 1 after inverting the corresponding sequence of $k$ destabilizations coming from each $K_i \subset M$, $\Phi^{+ K_i}_*$.  We then associate to this element $x^{i_1} \wedge \ldots \wedge x^{i_k}$ in $\Lambda_{\mathbb{F}}^k$, where $\varepsilon_{i_m}=0$ for $1 \leq m \leq k$.

By the definition of a hyperbox of chain complexes, the induced maps on $\partial$-homology, $\Phi^{+ K_i}_*$, commute; so, we see the order does not matter in this construction. Each exterior algebra element lives in the filtration level corresponding to the number of times we have applied a $(\Phi^{+ K}_*)^{-1}$.  However, since the destabilization maps lower the relative grading on $\C$ by 1 (these are components of the differential), we can see that each `element of $H^1(Y)$' has grading 1 and wedge product is additive on grading; furthermore, $U$ still has grading -2.  This therefore establishes the relatively-graded isomorphism of $E_3$ with $\fexterior \otimes \fpowers$.

Recall that $d_3$ lowers filtration level by 3, but it only lowers grading by 1.  In other words, $d_3$ will take $\alpha \otimes 1$, for $\alpha \in \Lambda_{\mathbb{F}}^i$, to $\beta \in \Lambda_{\mathbb{F}}^{i-3} \otimes U^j$ for some $j$.  By out identifications, the grading will be lowered by $3 + 2j$; therefore, $j$ must be $-1$.  Extending this gives $d_3: \Lambda^i_{\mathbb{F}}(Y) \otimes U^j \rightarrow \Lambda^{i-3}_{\mathbb{F}}(Y) \otimes U^{j-1}$, as desired.
\end{proof}

With these identifications, we are now ready to prove the first half of the main theorem.

\begin{proof}[Proof of First Three Differentials in Theorem~\ref{spectralcalc}]
We let $Y$ be presented by 0-surgery on a homologically split link $L$ with $\ell$ components.  Since there is a unique torsion Spin$^c$ structure on $Y$, $\mathfrak{s}_0$, and $\hfboldinfty(Y,\mathfrak{s}) = 0$ for non-torsion $\mathfrak{s}$, $\hfboldinfty(Y)=\hfboldinfty(Y,\mathfrak{s}_0)$.

Since there are three-manifolds with $b_1(Y) = 3$ where $\hfboldinfty$ does not have rank 8, we cannot repeat our arguments to show that all higher differentials vanish.  We will, however, be able to use the same argument to calculate $d_3$; see how it acts on faces and add up the components.  Note that we have identified the $E_3$ pages with $\fexterior \otimes \fpowers$ by associating $H_*(\C^\varepsilon,\partial)$ with span$_{\fpowers}\{x^{i_1} \wedge \ldots \wedge x^{i_k}\}$ where $\varepsilon_{i_m}=0$ for $1 \leq m \leq k = \|\varepsilon\|$.

We now can easily see how the subcomplexes and quotient complexes given by faces of the hypercube fit into this picture via inclusions/projections of $H^1$.  Let's prove that $d_3$ is given by (\ref{differential}).

Again, use $d^{i_1,i_2,i_3}_3$ to represent the components that map from $\C^\varepsilon$ to $\C^{\varepsilon + \tau_{i_1} + \tau_{i_2} + \tau_{i_3}}$ and let $x^{i_1}$ be the Hom-duals of the meridians of $K_i$ in $H^1$ (these naturally exist in $H^1$ of 0-surgery on any homologically split link containing $K_i$).  Consider the three-form, $\mu_{i_1,i_2,i_3}$, on $H^1(Y)$ given by $\langle x^{i_1} \smile x^{i_2} \smile x^{i_3},[S^3_{\mathbf{0}}(K_{i_1} \cup K_{i_2} \cup K_{i_3})]\rangle$  on $x^{i_1} \wedge x^{i_2} \wedge x^{i_3}$ and 0 otherwise.  The induction arguments with filtered morphisms of the previous section also show that $d^{i_1,i_2,i_3}_3$ will be given by interior multiplying $\mu_{i_1,i_2,i_3}$ if it does so for sublinks with $\ell-1$ components; this again follows by the injectivity on the $E_1 = E_3$ terms.

We claim that
\[
\mu_Y = \sum_{i_1 < i_2 <i_3} \mu_{i_1,i_2,i_3}.
\]
This is because the value of $\mu_Y$ on $x^{i_1} \wedge x^{i_2} \wedge x^{i_3}$ is given by the Milnor invariants of $L$, $\bar{\mu}_L(i_1,i_2,i_3)$ (see \cite{turaevcup}), and thus $\mu_{i_1,i_2,i_3}$ takes the value
$\bar{\mu}_{K_{i_1} \cup K_{i_2} \cup K_{i_3}}(i_1,i_2,i_3)$ on $x^{i_1} \wedge x^{i_2} \wedge x^{i_3}$.  Since $K_{i_1} \cup K_{i_2} \cup K_{i_3}$ is a sublink of $L$, their Milnor invariants agree (see, for example, \cite{murasugi}).

This follows by first establishing it on $\C_j$ by naturality of the inclusions of filtered complexes and then identifying $\C/\C_j$ with $\C_j$.

Thus, it suffices to establish that for the base case, $\ell = 3$, $d^{i_1,i_2,i_3}_3 = d_3$ is multiplication by some power of $U$ and $\langle x^1 \smile x^2 \smile x^3 ,[Y]\rangle$.  We can use the calculations of \cite{hfcoho} to do this without much effort.  The $E_3$ page has total rank 8.  However, in that paper, $\hfboldinfty$ is computed to have dimension $8 - 2 \cdot \langle x^1 \smile x^2 \smile x^3 ,[Y]\rangle$ .  Since $d_3$ can only be nonzero on $E_3^3$ by the depth of the filtration, we must study $E_3^3 = \Lambda_{\mathbb{F}}^3 \otimes \fpowers$ and $E_3^0 = \Lambda_{\mathbb{F}}^0 \otimes \fpowers$.  Each has rank 1, generated by $x^1 \wedge x^2 \wedge x^3$ and $1$ respectively in $\fexterior$.  Therefore, we may conclude that $d_3$ sends $x^1 \wedge x^2 \wedge x^3 \otimes U^j$ to $\langle x^1 \smile x^2 \smile x^3, [Y]\rangle \cdot 1 \otimes U^{j-1}$.

This shows that $d_3$ is given by contraction by the integral triple cup product form for $b_1 = 3$, completing the proof.
\end{proof}

Throughout the past two sections, we have only been proving facts about three-manifolds that can be expressed simply as 0-surgery on a homologically split link.  We now see that this was sufficient generality.

\begin{corollary} \label{surgeryequivalenceimplies} Let $\mathfrak{s}$ be a torsion Spin$^c$ structure on a closed, connected, oriented three-manifold $Y$.  Then, $\operatorname{dim}_{\fpowers}\hfboldinfty(Y,\mathfrak{s}) \leq \operatorname{dim}_{\fpowers}\hcinfty(Y)$.  Therefore, the rank of $\hfbold^\infty(Y,\mathfrak{s})$ is at most that of the rank conjectured by Ozsv\'ath and Szab\'o.
\end{corollary}

\begin{proof}
Since the homology predicted by Ozsv\'ath and Szab\'o is isomorphic to $\hcinfty(Y)$, Theorem~\ref{spectralcalc} implies that this inequality of ranks is the case for any three-manifold given by 0-surgery on a homologically split link.  This is because we are comparing two spectral sequences agreeing up to $(E_3,d_3)$, where the higher differentials of one all vanish.  Now let $Y$ be arbitrary.  Choose a homologically split link, $L$, with the property that 0-surgery has integral triple cup product form isomorphic to that of $Y$.  By the work of Cochran, Gerges, and Orr \cite{surgeryequivalence}, such a link exists and $\hcinfty(Y) \cong \hcinfty(S^3_{\mathbf{0}}(L))$.  Furthermore, we know from \cite{hfcoho} that $\hfboldinfty(Y,\mathfrak{s}) \cong \hfboldinfty(S^3_{\mathbf{0}}(L))$.  This proves the result.
\end{proof}




The rest of this paper is now devoted to showing the higher differentials in our spectral sequence vanish, or equivalently, proving the opposite inequality:
\begin{equation} \label{rankinequality}
\operatorname{dim}_{\fpowers} \hfboldinfty \geq \operatorname{dim}_{\fpowers} \hcinfty \tag{+}
\end{equation}

\section{Composing Knots and Complexities of Links}

In order to prove (\ref{rankinequality}), it is necessary to again proceed inductively.  However, we must induct on something more complicated than simply $b_1$.  We will assume that the higher differentials vanish when $b_1 \leq \ell-1$ (this is automatic for $b_1 \leq 3$), but we will need a way to seemingly induct on the set of homologically split links with $\ell$ components.  Recall that two three-manifolds are surgery equivalent if there is a finite sequence of $\pm 1$-surgeries on nullhomologous knots taking one manifold to the other (see \cite{surgeryequivalence}).  As mentioned before, we only need to prove (\ref{rankinequality}) for a single representative of each surgery equivalence class of three-manifold with $H_1(Y) \cong \mathbb{Z}^\ell$ (Section 2 of \cite{hfcoho}); we will take this liberty and change our links around to make them more suitable to the link surgery formula.

All components will have framing 0, so we will not distinguish between the link and the resulting manifold obtained by 0-surgery, or between $b_1$ and the number of components.  Thus, we will make statements like surgery equivalent links to mean that the manifolds obtained by 0-surgery on each link are surgery equivalent.  Also, $\mathfrak{s}_0$ will always refer to the unique torsion Spin$^c$ structure on the underlying manifold.

We let $L_1 \coprod L_2$ indicate that the two links are separated by an embedded 2-sphere (and both links will always be nonempty when using this notation).  Begin with an $\ell$-component homologically split link $L$.  Order the components $K_1,\ldots, K_\ell$; we refer to the Milnor linking invariants as $\bar{\mu}_L(i,j,k)$ or $\bar{\mu}_L(K,K',K'')$ if the indices are unclear.  We always assume that the three indices are all distinct.  We will also usually assume $i<j<k$, but may switch this at the reader's inconvenience to keep notation simple.  Note that changing the order of the indices does not change $\bar{\mu}$ mod 2.

\begin{example} \label{c(L)=1} Suppose that $\bar{\mu}_L(1,2,3) = n$ and all $\bar{\mu}$ vanish for all other triples of indices.  Let $L' = K_1 \cup K_2 \cup K_3$.  Then we know that $L$ is surgery equivalent to a link of the form $(L-L') \coprod L'$ \cite{surgeryequivalence}.  Given (\ref{rankinequality}) for all links with at most $\ell-1$ components, the connect-sum formula will guarantee that $\hfboldinfty(S^3_{\mathbf{0}}(L)) \cong \hfboldinfty(S^3_{\mathbf{0}}(L')) \otimes \hfboldinfty(S^3_{\mathbf{0}}(L-L'))$; since this formula also holds for $\hcinfty$ with $\fpowers$ coefficients (see the proof of Theorem 2 in \cite{thommark}), this proves Theorem~\ref{maintheorem} for $L$ as well.
\end{example}

With this example in mind, we define a complexity of $L$ with the hope that a reduction in complexity makes the link closer to being split.  This is defined as
\begin{equation}
c(L) = \# \{(i,j,k): \bar{\mu}_L(i,j,k) \neq 0\text{, } 1\leq i<j<k \leq \ell\}
\end{equation}

Let's study the links with the simplest $c$-complexity first.

\begin{remark} \label{easycomplexity}  If $c(L) = 0$ or $c(L)=1$, then we know how to complete the proof from Example~\ref{c(L)=1}.  If $c(L) \geq 2$, then there are two options.  Either $L$ is surgery equivalent to some $L_1 \coprod L_2$ or there exists some component $K_i$ which has $\bar{\mu}_L(i,j,k)$ nonzero for at least two different pairs $(j,k)$ (reordering of $(i,j,k)$ possibly necessary).  If it splits as $L_1 \coprod L_2$, then again we are done by the connect-sum formulae.
\end{remark}

For a fixed $b_1$ we will induct on the $c$-complexity.  To keep sight of the final goal, the plan for the rest of the paper will be to prove the following theorem similar to the method of composing knots in \cite{hfcoho}.

\begin{theorem} \label{connectsumimplies} Suppose $K = K' \# K''$ is a component of $L$.  If (\ref{rankinequality}) holds for $(L-K) \cup K'$ and $(L-K) \cup K''$, then it will hold for $L$.
\end{theorem}

We are therefore led to the following proposition.

\begin{proposition} \label{splitknots}
Suppose $c(L) \geq 2$ and that $L$ is not surgery equivalent to any $L_1 \coprod L_2$.  Let $K_r$ be a component with at least two different pairs $(s,t)$ such that $\bar{\mu}_L(r,s,t)$ are nonzero.  Then, there is an ordered $\ell$-component link $\tilde{L}$ with the following two properties.  First, $\bar{\mu}_L(i,j,k) = \bar{\mu}_{\tilde{L}}(i,j,k)$ for all $i,j,k$.  Second, there is a knot $K \subset \tilde{L}$ which we can express as $K' \# K''$, where $c((\tilde{L}-K) \cup K')<c(L)$ and $c((\tilde{L}-K) \cup K'')<c(L)$.
\end{proposition}

Before proving Proposition~\ref{splitknots}, we now see how this will be applied in conjunction with Theorem~\ref{connectsumimplies}.

\begin{proof}[Proof of Theorem~\ref{maintheorem}]
For fixed $b_1=\ell$, we induct on $c$. By Remark~\ref{easycomplexity}, we need only concern ourselves with the case where $c(L) \geq 2$ and $L$ does not decompose as two geometrically split links.  Apply the proposition to replace $L$ by $\tilde{L}$.  Since $\bar{\mu}_L(i,j,k)=\bar{\mu}_{\tilde{L}}(i,j,k)$, $L$ and $\tilde{L}$ will be surgery equivalent.  It now suffices to prove (\ref{rankinequality}) on $\tilde{L}$.

Decompose the component $K$ as $K' \# K''$.  Since $\tilde{L} - K  \cup K'$ and $\tilde{L} - K \cup K''$ have strictly smaller $c$-values, (\ref{rankinequality}) holds for each of these.  Theorem~\ref{connectsumimplies} completes the proof of (\ref{rankinequality}) for all $b_1 = \ell$.  Applying Theorem~\ref{spectralcalc} shows that for 0-surgeries on homologically split links $\hfboldinfty$ and $HC^\infty$ have the same rank; the relative gradings also agree simply because the complexes $(E_3,d_3)$ and $C_*^\infty$ agree on relative gradings.  Having this for 0-surgeries on homologically split links gives the relatively-graded isomorphism for arbitrary three-manifolds by repeating the arguments from the proof of Corollary~\ref{surgeryequivalenceimplies}.
\end{proof}

We now recall a helpful theorem of Cochran describing the $\bar{\mu}$-invariants of connect sums.

\begin{theorem} \label{connectsumadds}(Theorem 8.13 of \cite{cochranderivatives}) Suppose $L$ and $L'$ are $\ell$-component links that are separated by an embedded 2-sphere and that $\bar{\mu}_L(J) = \bar{\mu}_{L'}(J) = 0$ for multi-indices $J$ of length at most $\ell$.  Construct $L \# L'$ by connecting each pair of components $L_i$ and $L'_i$ with a band that passes through the separating sphere exactly once.  Then $\bar{\mu}_{L \# L'}(I) = \bar{\mu}_{L}(I) + \bar{\mu}_{L'}(I)$ for any multi-index $I$ of length at most $\ell+1$.
\end{theorem}

Therefore, in the case of two homologically split links, $\bar{\mu}_{L \# L'}(i,j,k) = \bar{\mu}_L(i,j,k) + \bar{\mu}_{L'}(i,j,k)$.

\begin{proof} [Proof of Proposition~\ref{splitknots}] By hypothesis, we may consider two distinct pairs $(j_1,k_1)$ and $(j_2,k_2)$ such that $\bar{\mu}_L(r,j_1,k_1)$ and $\bar{\mu}_L(r,j_2,k_2)$ are nonzero for $L$.  Construct an $\ell$-component homologically split link $L'$ with an ordering on the components such that

\begin{equation*}
\bar{\mu}_{L'}(a,b,c) = \left\{
\begin{array}{rl}
\bar{\mu}_L(a,b,c) & \text{if } (a,b,c) \neq (r,j_2,k_2),\\
0 & \text{if } (a,b,c) = (r,j_2,k_2).
\end{array} \right.
\end{equation*}
Such a link can be explicitly constructed by repeated applications of ``Borromean braiding" (see Corollary 3.5 of \cite{surgeryequivalence} for more details).

Next, isotope a small arc from both $K'_{j_2}$ and $K'_{k_2}$ out and away from the rest of the diagram for $L'$, and isotope the arc from $K'_{j_2}$ such that it creates $\bar{\mu}_L(r,j_2,k_2)$ twists.  We now take an unknot, $U$, and thread it through the twists of $K'_{j_2}$ and through $K'_{k_2}$ as in Figure~\ref{threadknot}.  This three-component sublink $(U,K'_{j_2},K'_{k_2})$ has Milnor invariants equal to $\bar{\mu}_L(r,j_2,k_2)$.

\begin{figure}[ht!]
\labellist
\small
\pinlabel $K'_{j_2}$ at 75 180
\pinlabel $K'_{k_2}$ at 650 220
\pinlabel $U$ at 375 20
\scriptsize
\pinlabel $\bar{\mu}_L(r,j_2,k_2)$ at 310 280
\pinlabel twists at 312 260
\endlabellist
\qquad\qquad\quad\includegraphics[scale=.34]{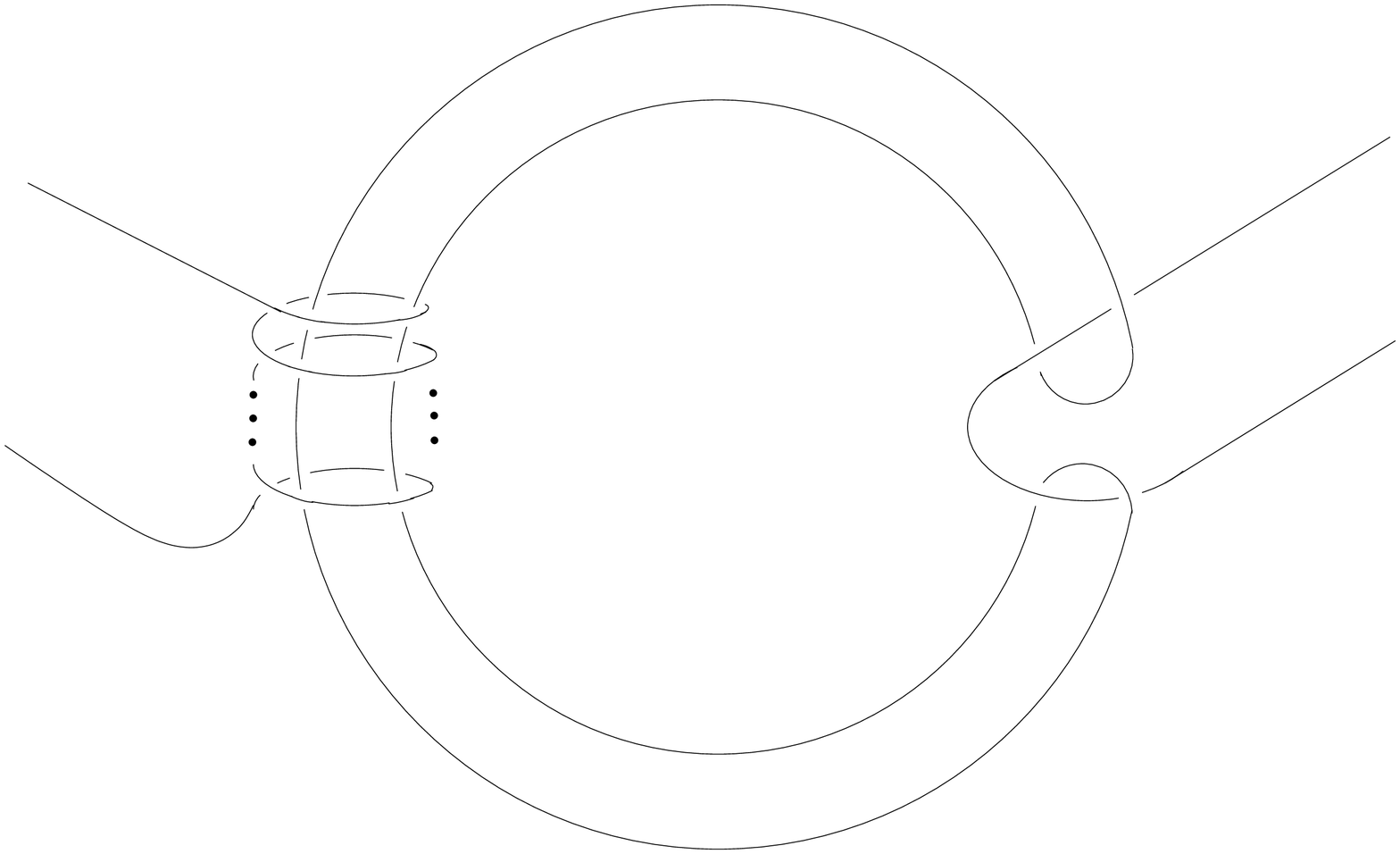}
\caption{Threading the unknot to recreate $\bar{\mu}_L(r,j_2,k_2)$} \label{threadknot}
\end{figure}

We will choose $K'$ to be $K'_r$ in $L'$ and $K'' = U$.  From this we can see that $\bar{\mu}_{(L'-K') \cup K''}(K'',K'_{j_2},K'_{k_2}) = \bar{\mu}_L(r,j_2,k_2)$ and all other $\bar{\mu}_{(L'-K') \cup K''}(K'',\cdot,\cdot)$ vanish.  We now want to see that the connect sum $K=K' \# K''$ yields a link $\tilde{L} = (L' - K') \cup K$ with all $\bar{\mu}_{\tilde{L}}(a,b,c) = \bar{\mu}_L(a,b,c)$.

To show this, it suffices to prove that $\tilde{L}$ can be constructed by connecting two geometrically split $\ell$-component links by bands between pairs of components which intersect the separating 2-sphere exactly once; furthermore, we require that the $\bar{\mu}$-invariants for these two links add up to $\bar{\mu}_L$.  The result will then follow from the additivity of $\bar{\mu}$ in Theorem~\ref{connectsumadds}.

We choose our two links as follows.  The first link will be $L'$.  The other link is an $\ell$-component link, $L^*$, consisting of two split sublinks: a three-component homologically split sublink with $\bar{\mu}_{L^*}(r,j_2,k_2) = \bar{\mu}_L(r,j_2,k_2)$ and an $(\ell-3)$-component unlink, so all other invariants vanish.  Clearly the values of $\bar{\mu}$ add up as expected and Figure~\ref{linksum} demonstrates how we can connect them to obtain $\tilde{L}$ with $K = K' \# K''$.  Note that $L^*_r$ is what creates $K''$ in $\tilde{L}$.

\begin{figure}[ht!]
\labellist
\small
\pinlabel $L'$ at 355 131
\pinlabel $L^*$ at 615 128
\pinlabel $L^*_r$ at 545 475
\pinlabel $L^*_{j_2}$ at 860 430
\pinlabel $L^*_{k_2}$ at 860 535
\pinlabel $K'_{k_2}$ at 210 590
\pinlabel $K'_{j_2}$ at 430 560
\pinlabel $K'_{j_1}$ at 110 325
\pinlabel $K'_{k_1}$ at 295 270
\pinlabel $K'_r$ at 135 440
\pinlabel $S^2$ at 490 70
\endlabellist
\quad\includegraphics[scale=.34]{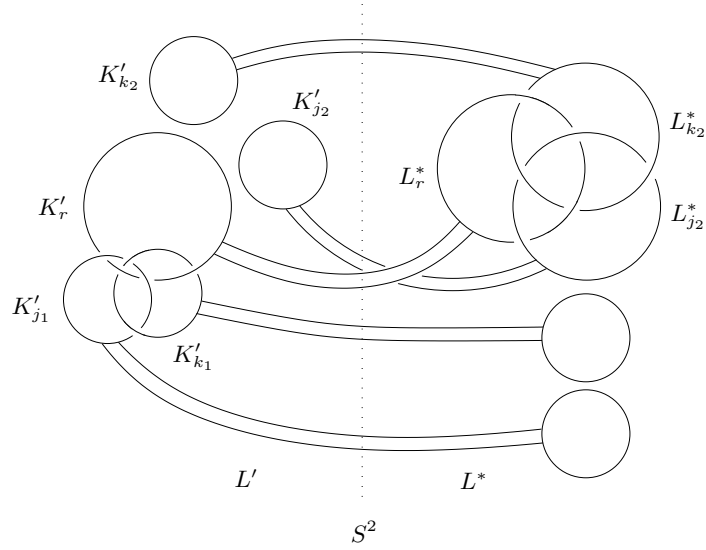}
\caption{Expressing $\tilde{L}$ as the connect-sum of $L'$ and $L^*$} \label{linksum}
\end{figure}

By construction, both $c((\tilde{L}-K) \cup K')$, which equals $c(L')$, and $c((\tilde{L}-K) \cup K'')$ are strictly less than $c(L)$.  This completes the proof.
\end{proof}

\begin{remark}
Because $L$ and $\tilde{L}$ produce indistinguishable cup homology and $\hfboldinfty$, we will use $L$ to in fact refer to $\tilde{L}$ for the remainder of the paper.
\end{remark}

\section{Constructing and Chopping Down the Complex}
Recall that our goal is to prove Theorem~\ref{connectsumimplies}; this loosely said that knowing the higher differentials vanish for a link with a component replaced by $K'$ or $K''$, then this holds after instead replacing with $K' \# K''$.  We now construct a complex which contains all of the Heegaard Floer information of $K'$ and $K''$ simultaneously, where we have identified $K = K' \# K''$ as the component to reduce complexity at.  With this we will be able to use our inductive knowledge for $K'$ and $K''$ to produce the desired result.  The way that this is done is via a standard Kirby calculus trick (see, for example, \cite{saveliev}); we express 0-surgery on $K$ as 0-surgery on three components: $K'$, $K''$, and an unknot $U$ geometrically linking each once as shown in Figure~\ref{kirbydiagram}.

\begin{figure}[h]
\labellist
\small
\pinlabel $U$ at 95 245
\pinlabel $0$ at 275 180
\pinlabel $K'$ at 57 35
\pinlabel $0$ at 165 180
\pinlabel $K''$ at 380 35
\pinlabel $0$ at 220 80
\endlabellist
\qquad\qquad\qquad\quad\quad\includegraphics[scale=.42]{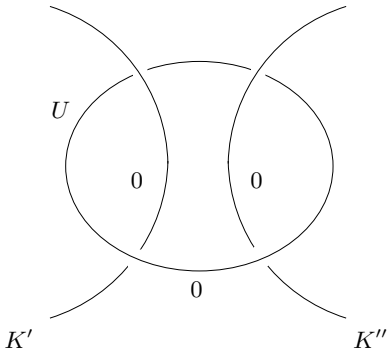}
\caption{An equivalent diagram for 0-surgery on $K' \#K''$} \label{kirbydiagram}
\end{figure}

Thus, if our link has $\ell$ components, then the corresponding link that we would like to study will have $\ell+2$.  As further abuse of notation we will now call this link $L$, since 0-surgery results in the same manifold.  The framing $\Lambda$ will change as well due to the algebraic linking that has been introduced.  Reorder the components in such a way that $K'$, $K''$, and $U$ are the first, second, and third components respectively.  For notational purposes, we will relabel these as $K_1$, $K_2$, and $K_3$.  This three-component sublink will arise often, so we will refer to it as $W$.  We see that $\Lambda_1 = \Lambda_2 = (0,0,1,0, \ldots, 0)$ and $\Lambda_3 = (1,1,0,\ldots,0)$.  Therefore, the equivalence class in $\mathbb{H}(L)$ corresponding to $\mathfrak{s}_0$ will in fact be a 2-dimensional lattice spanned by $\Lambda_1$ and $\Lambda_3$; in fact, $\mathfrak{s}_0 = [(\frac{1}{2},\frac{1}{2},1,0,\ldots,0)]$.

By inducing the proper filtrations and removing acyclic complexes, we will significantly cut down the size of $\C^\infty(\hyperbox,\Lambda,[(\frac{1}{2},\frac{1}{2},1,0,\ldots,0)])$ to a smaller finite-dimensional object.
Before continuing, we remark that the reader interested in this proof should first try to follow the calculation of $\mathbf{HF}^-$ for surgeries on the Hopf link via the link surgery formula (Section 8.1 of \cite{hflz}), as the arguments will be based on this.  Let's also recall the simplified notation used in that computation.  We let ${\varepsilon_1 \varepsilon_2 \hdots \varepsilon_{\ell+2}}_{\mathbf{s}}$ represent the complex $\mathfrak{A}^\infty(\hyperbox^{L-M},\psi^M(\mathbf{s}))$ where $K_i$ is in $M$ if and only if $\varepsilon_i=1$.  To shorten notation further in our setting, we will use $\varepsilon_1 \varepsilon_2 \varepsilon_3 *_{(s_1,s_2,s_3)}$ to denote the hypercube of chain complexes at $(s_1,s_2,s_3,0,\ldots,0)$ with $\varepsilon_1$, $\varepsilon_2$, $\varepsilon_3$ fixed, but all remaining $\varepsilon_i$ free.  We are setting the last components of $\mathbf{s}$ to be 0 since this corresponds to choosing the unique torsion Spin$^c$ structure on $S^3_{\mathbf{0}}(L-W)$.  In fact, $111*_{(s_1,s_2,s_3)}$ is exactly the complex corresponding to 0-surgery on $L - W$ with the torsion Spin$^c$ structure.

A key map that we will study is $\Gamma^{\pm K_i} = \sum_{\vec{N} \subset L-W} \Phi^{\pm K_i \cup \vec{N}}$.  Here we are summing over all possible orientations of sublinks $N$, whereas the previous maps $\mathcal{S}^{+K_i}$ only allowed for $\vec{N}=+N$.

\begin{proposition} \label{finalshape} The complex for 0-surgery on all components in $L$ with Spin$^c$ structure $\mathfrak{s}_0$, $\C=\C^\infty(\hyperbox,\Lambda,[(\frac{1}{2},\frac{1}{2},1,\ldots,0)])$, is quasi-isomorphic to

\[
\xymatrix{
001*_{(\frac{1}{2},\frac{1}{2},1)} \ar[rr]^{\Gamma^{+K_1}} \ar[rrd]^{\Gamma^{+K_2}} && 101*_{(\frac{1}{2},\frac{1}{2},1)} && 001*_{(\frac{1}{2},\frac{1}{2},0)} \ar[ll]_{\Gamma^{-K_1}} \ar[lld]_{\Gamma^{-K_2}} \\	&& 011*_{(\frac{1}{2},\frac{1}{2},1)} \\					
}
\]
\end{proposition}
\begin{proof}
From now on, a complex $\{s_i > r\}$ will refer to all ${\varepsilon_1 \varepsilon_2 \varepsilon_3 *}_{(s_1,s_2,s_3)}$, where $s_i > r$; this is regardless of whether the component $K_i$ has been destabilized or what the value of $s_i$ is under some subsequent $\psi^M$ maps.  For this reason, we omit the $\mathbf{s}$ from the $\Phi$ maps.

Induce the filtration on $\C$ defined by $\F_3(x) = -(s_1 + \sum_{i \neq 3} \varepsilon_i)$ for $x \in {\varepsilon_1 \varepsilon_2 \varepsilon_3 *}_{(s_1,s_2,s_3)}$.  The components of the differential that preserve filtration level are given by $\partial$ and $\Phi^{+K_3}$.  Consider the subcomplex $\{s_1 > \frac{1}{2} \}$.    The associated graded with respect to the filtration on the subcomplex splits as a product of complexes of the form
\[
\xymatrix{
({\varepsilon_1,\varepsilon_2,0,\varepsilon_4,\ldots,\varepsilon_{\ell+2}}_{\mathbf{s}},\partial) \ar[r]^{\Phi^{+K_3}} & ({\varepsilon_1,\varepsilon_2,1,\varepsilon_4,\ldots,\varepsilon_{\ell+2}}_{\mathbf{s}},\partial) \\
}
\]

Since the maps $\Phi^{+K_3}$ are quasi-isomorphisms, we have that the associated graded, and thus all of $\{s_1 > \frac{1}{2} \}$, is acyclic.  Therefore, $\C$ is quasi-isomorphic to the quotient complex $\C/\{s_1 > \frac{1}{2} \}$, which is $\{s_1 \leq \frac{1}{2} \}$.  We then induce a similar filtration, $\mathcal{G}_3(x) = s_1 - \sum_i\varepsilon_i$.  The differentials preserving the filtration level will now be $\partial$ and $\Phi^{-K_3}$.  We consider the subcomplex, $\C '$, of $\{s_1 \leq \frac{1}{2} \}$ defined by $\{s_1 < \frac{1}{2}, \varepsilon_3 = 0 \} \oplus \{s_1 \leq \frac{1}{2}, \varepsilon_3 = 1\}$ (this is everything except $\{s_1 = \frac{1}{2}, \varepsilon_3 = 0 \}$).  This subcomplex now splits as a product of complexes of the form
\[
\xymatrix{
({\varepsilon_1,\varepsilon_2,0,\varepsilon_4,\ldots,\varepsilon_{l+2}}_{\mathbf{s}},\partial) \ar[r]^{\Phi^{-K_3}} & ({\varepsilon_1,\varepsilon_2,1,\varepsilon_4,\ldots,\varepsilon_{l+2}}_{\mathbf{s}+\Lambda_3},\partial) \\
}
\]
Similarly, since the maps $\Phi^{-K_3}$ are quasi-isomorphisms, $\C '$ is acyclic.  We are content to remove this and study only the remaining terms, namely $\{s_1 = \frac{1}{2}, \varepsilon_3 = 0\}$.  We have essentially collapsed this complex in the $\Lambda_3$-direction.
It is best to visualize the remaining complex via Figure~\ref{complexfigure}.

\begin{figure}

\[
\begin{matrix}
&&& \vdots \\ \\
&& 000* & 100* & \mbox{110$*$} & (\frac{1}{2},\frac{1}{2},1) \\
&&      & 010* & \\ \\
&& 000* & \mbox{100$*$} & \mbox{110$*$}  & (\frac{1}{2},\frac{1}{2},0) \\
&&      & \mbox{010$*$} & \\ \\
&& \mbox{000$*$} & \mbox{100$*$} & \mbox{110$*$} & (\frac{1}{2},\frac{1}{2},-1) \\
&&      & \mbox{010$*$} & \\ \\
\bigg{\uparrow}^{\Lambda_1 = \Lambda_2}&& & \vdots \\
\end{matrix}
\]
\caption{The complex $\{s_1 = \frac{1}{2}, \varepsilon_3 = 0\}$}\label{complexfigure}
\end{figure}

We can further reduce this complex in a similar way, by collapsing in the $\Lambda_1$-direction.  Consider the filtration, $\F_1(x) = -(s_3 + \sum_{i \neq 1} \varepsilon_i)$, on the subcomplex $\{s_3 > 1 \}$  of $\{s_1 = \frac{1}{2}, \varepsilon_3 = 0\}$.  The associated graded splits as a product of
\[
\xymatrix{
({0,\varepsilon_2,0,\varepsilon_4,\ldots,\varepsilon_{l+2}}_{\mathbf{s}},\partial) \ar[r]^{\Phi^{+K_1}} & ({1,\varepsilon_2,0,\varepsilon_4,\ldots,\varepsilon_{l+2}}_{\mathbf{s}},\partial) \\
}
\]

This complex is acyclic, as the $\Phi^{+K_1}$ are quasi-isomorphisms.  After removing this subcomplex, we are left with $\{s_3 \leq 1 \}$.  Since $\Lambda_1 = \Lambda_2$, we cannot repeat the argument for $\Phi^{-K_3}$ to remove $\{s_3 <1\}$.  We must tread carefully to chop the remaining complex down further.  Consider the filtration, $\F_2(x) = s_3 - 2\varepsilon_1 - \sum_{i \neq 1} \varepsilon_i$.  This odd-looking filtration is defined such that $\Phi^{-K_1}$ lowers the filtration level, but $\Phi^{-K_2}$ does not, even though $\Lambda_1 = \Lambda_2$.  We now study the subcomplex $\{s_3 = 1, \varepsilon_1 = \varepsilon_2 = 1\} \oplus \{s_3 = 0, \varepsilon_1+\varepsilon_2 \geq 1\} \oplus \{s_3 \leq -1 \}$.  This subcomplex is best seen by the boxed elements in Figure~\ref{boxedcomplex}.

\begin{figure}
$$ \;\overline{\qquad\qquad\qquad\qquad\quad} $$

\[
\begin{matrix}
&& 000* & 100* & \fbox{110$*$} & (\frac{1}{2},\frac{1}{2},1) \\
&&      & 010* & \\ \\
&& 000* & \fbox{100$*$} & \fbox{110$*$}  & (\frac{1}{2},\frac{1}{2},0) \\
&&      & \fbox{010$*$} & \\ \\
&& \fbox{000$*$} & \fbox{100$*$} & \fbox{110$*$} & (\frac{1}{2},\frac{1}{2},-1) \\
&&      & \fbox{010$*$} & \\ \\
&\bigg{\uparrow}^{\Lambda_1 = \Lambda_2}& & \vdots \\ \\
\end{matrix}
\]
\caption{The boxed terms form the final acyclic complex} \label{boxedcomplex}
\end{figure}

The associated graded splits into a product of complexes analogous to the ones defined previously.  Since $\Phi^{-K_2}$ is a quasi-isomorphism, we may again remove this acyclic complex in our study.  We can now see that the remaining complex is the same as the one in the statement of the proposition, except for the fact that all $\varepsilon_3$ are 0 instead of 1.  However, we can simply apply the map $\mathcal{S}^{+K_3} = \sum_{M \subset L'} \Phi^{+K_3 \cup +M}$ to all of the components to obtain a filtered quasi-isomorphism between the complex with $\varepsilon_3 = 0$ and the one with $\varepsilon_3 = 1$ by Lemma~\ref{pageiso}.  This completes the proof.
\end{proof}

\section{Reshaping the Complex} \label{gluing}

We have actually reduced the computation to calculating the homology of the complex given by

\[
\xymatrix{
00*_{(0,0)} \ar[rr]^{\Gamma^{+K_1}} \ar[rrd]^{\Gamma^{+K_2}} && 10*_{(0,0)} && \ar[ll]_{\Gamma^{-K_1}} \ar[lld]_{\Gamma^{-K_2}} 00*_{(0,0)}\\
&& 01*_{(0,0)} \\					
}
\]

Here we have suppressed the $\varepsilon_3$-coordinate by destabilizing this component via $\mathcal{S}^{+K_3}$ and applying Lemma~\ref{removecomponent} (it is easy to see this applies to the truncated complex as well).  Because of this, we have applied $\psi^{+K_3}$ to all $\mathbf{s}$.  Since all values of $\mathbf{s}$ and $\psi^{M}(\mathbf{s})$ are now all $\mathbf{0}$, we will suppress this for the remainder of the proof.

To simplify this problem further, we need some simple linear algebra.  Given a mapping complex over a vector space, $M(\phi:(V,\partial_V) \rightarrow (W,\partial_W))$, to determine the rank of $H(M(\phi))$, we only need to know the homologies of $V$ and $W$ and $\operatorname{rk} \phi_*$.  In the case that $H(V) \cong H(W)$ we have the convenient formula
\[
\operatorname{dim} H(M(\phi)) = 2 \operatorname{dim} H(V) - 2 \operatorname{rk} \phi_*,
\]
where this is a statement about the $\fpowers$-dimensions of the entire homology vector spaces, not at each grading.
This follows easily from the long exact sequence associated to $H(V)$,$H(W)$, and $H(M(\phi))$.  We now apply this to prove a convenient lemma.
\begin{lemma} \label{gluinglemma} Suppose $V$ is a finite-dimensional vector space over a field of characteristic 2.  Consider the complex given by $V$ equipped with the differential $\partial \equiv 0$.  Let $F,G,J,K:V \rightarrow V$, and define $\Theta:V \oplus V \rightarrow V \oplus V$ by
\begin{equation*}
\Theta(v,w) = (F(v) + G(w), J(w) + K(v)).
\end{equation*}
Furthermore, suppose that $J$ is a quasi-isomorphism (or equivalently, an invertible map).  Then, the homology of the mapping cone, $M(\Theta)$, has the same dimension as the homology of $M(F-GJ^{-1}K)$.
\end{lemma}

\begin{proof}
We know that the homology of $M(\Theta)$ has rank given by $2 \operatorname{dim} (V \oplus V) - 2 \operatorname{rk} \Theta$.  We study the matrix
$\Theta = \begin{pmatrix}
F & G \\
K & J
\end{pmatrix}$.
It is easy to see that this matrix has the same rank as
$X = \begin{pmatrix}
F-GJ^{-1}K & 0 \\
K & J
\end{pmatrix}$.
Now, we have
\begin{align*}
\operatorname{dim} H(M(\Theta)) =& 2 \operatorname{dim} (V \oplus V) - 2 \operatorname{rk} X \\ =& 4 \operatorname{dim} V - 2(\operatorname{rk}(F-GJ^{-1}K) + \operatorname{dim} V) \\
=& 2 \operatorname{dim} V - 2\operatorname{rk}(F-GJ^{-1}K) \\ =& \operatorname{dim} H(M(F-GJ^{-1}K)). \qedhere
\end{align*}
\end{proof}

In order to apply Lemma~\ref{gluinglemma} to Proposition~\ref{finalshape}, we must see that one of the maps $\Gamma$ is a quasi-isomorphism.  In fact, it is easy to see each $\Gamma$ map is a quasi-isomorphism by applying the same filtration arguments as in the proof of Proposition~\ref{finalshape}.  Thus, it remains to calculate the rank of the induced map $\Gamma^{+K_1}_* + \Gamma^{-K_1}_* \circ (\Gamma^{-K_2}_*)^{-1} \circ \Gamma^{+K_2}_*$ from $H_*(00*)$ to $H_*(10*)$.  For notational convenience, we will abbreviate this induced map by $\Psi^{K_1,K_2}$.

\section{The Maps $\Gamma^{\pm K_i}$}
In order to study the $\Gamma$ maps, it is useful to note that the homology of each complex, $00*$, $10*$, or $01*$, is naturally isomorphic to $\hfboldinfty(S^3_{\mathbf{0}}(L'))$ by Lemma~\ref{removecomponent}, where $L'$ is $L - W$.  Recall that $L'$ consisted of one less component than the level of $b_1$ that we wanted for Theorem~\ref{maintheorem}.  Therefore, we may assume that this homology is exactly $\hcinfty(S^3_{\mathbf{0}}(L'))$.   However, we want to study this a little more carefully.

Let $x^1$ represent the Hom-dual of the class $[K_3]$ in $H^1(S^3_{\mathbf{0}}(L' \cup K_i))$.  Choose the other basis vectors of $H^1$ to be given by the meridians of the components of $L'$.  Using the $\varepsilon$-filtration restricted to the $*$-part of the complex, we may think of the $E_3$ term for $00*$ as $x^1 \wedge \fexterior(S^3_{\mathbf{0}}(L')) \otimes \mathbb{F}[[U,U^{-1}]$ and the $E_3$ term of $10*/01*$ as simply $\fexterior(S^3_{\mathbf{0}}(L')) \otimes \mathbb{F}[[U,U^{-1}]$.  Therefore, we will think of taking homology of $00*$ as taking $x^1$ and wedging with the elements of $\hcinfty(S^3_{\mathbf{0}}(L'))$.

The homology of each complex $00*/10*/01*$ is actually the homology of $E_3$ with respect to $d_3-d^{K_1}_3$ thought of as a differential on $\cinfty(S^3_{\mathbf{0}}(L'))$.  We mean by $d^{K_1}_3$ the components of $d_3$ that correspond to contracting by triple cup products that contain $x^1$.  Therefore, we can consider the maps $d^{K_i}_3: (x^1 \wedge \fexterior(S^3_{\mathbf{0}}(L'))) \otimes U^j \rightarrow \fexterior(S^3_{\mathbf{0}}(L')) \otimes U^{j-1}$. However, $d^{K_i}_3$ is in fact a chain map with respect to the differential $d_3 - d^{K_i}_3$ on these two complexes since $\iota_\alpha \circ \iota_\beta = \iota_{\alpha \wedge \beta} = \iota_\beta \circ \iota_\alpha$ (mod 2) (the composition of interior multiplying by two different trilinear forms corresponds to interior multiplying by the wedge product).  Therefore, it makes sense to talk about $(d^{K_i}_3)_*$.

Finally, by choosing a basis with the meridian of $K$ to be $x^1$ and the other basis vectors as meridians of the components of $L'$, we can identify $\fexterior(S^3_{\mathbf{0}}(L))$ with each $\fexterior(S^3_{\mathbf{0}}(L' \cup K_i))$.  Thus, we can make sense of the statement $d^K_3 = d^{K_1}_3 + d^{K_2}_3$.  Better yet, this statement is actually true by the construction of $K = K_1 \# K_2$ in Proposition~\ref{splitknots}.

What we hope to find is that the map $\Psi^{K_1,K_2}$ is precisely $(d^K_3)_*$ from $H_*(00*)$ to $H_*(10*)$, so as to guarantee that the higher differentials must vanish.  It turns out that this is not the case, but it will be true up to some terms of higher order.

We now give a brief outline for how the rest of the proof is going to go.  First, we will set up the complex so that $\Gamma^{+K_i}_*$ will essentially be given by the identity and $\Gamma^{-K_i}_*$ by the identity plus a term that corresponds to $(d^{K_i}_3)_*$.  Therefore, $\Psi^{K_1,K_2}$ will be $(d^{K_1}_3)_* + (d^{K_2}_3)_* + (d^{K_1}_3)_* \circ (d^{K_2}_3)_*$.  Finally, it will be a simple slight of hand to prove from here that the rank of $\hfboldinfty$ is at least that of the predicted homology, $HC^\infty$.

From now on we will use $00*$ (and similar complexes) to refer to its homology, as well as for the maps $\Gamma^{\pm K_i}$, unless specified otherwise.  A mindful reader may have noticed that while we have been using hyperboxes of Heegaard diagrams to make all of the constructions so far, there have been no restrictions on the choice of complete system.  Now is the time where we do so.

The complete system of hyperboxes that we will work with is a basic system of hyperboxes.  Instead of recalling the construction, we will review only the properties we will use and refer the reader to Section 6.7 in \cite{hflz}.  Basic systems have the property that if $\vec{M'}$ has the induced orientation of $\vec{M}$ for a sublink $M' \subset M$, then $\hyperbox^{\vec{M},\vec{M'}}$ consists of a single Heegaard diagram and $\hyperbox^{\vec{M}-\vec{M'}}$ is obtained from $\hyperbox^{\vec{M}}$ by removing the $z$ basepoints corresponding to components of $M'$.  Let $K$ be the $i$th component of $L$.  By compatibility, a hyperbox $\hyperbox^{+K \cup \vec{M},\vec{M'}}$ has $d_i = 0$ (it has 0 in the $i$th component of the size).

\begin{lemma} \label{higherhomotopiesvanish} Suppose we are working in a basic system.  If $\vec{M}$ has at least two components, one of which is compatibly oriented with $L$, then $\Phi^{\vec{M}}$ vanishes; in other words, for all $K$, $\Phi^{+K} = \mathcal{S}^{+K} = \Gamma^{+K}$.
\end{lemma}
\begin{proof} Let $\vec{M'}$ be a nonempty sublink of $L - K$ and suppose that $K$ is the $i$th component, consistently oriented.  We will show that $\destab^{+K \cup \vec{M'}}$ vanishes.  Since $\Phi = \destab \circ \proj$, this will prove the lemma.

Let's study destabilization maps more carefully.  Destabilizing a link of $k$ components is given by compressing the hyperboxes, or in other words, playing the $k$th standard symphony for some hypercubical collection (see Section 3 of \cite{hflz}); if one of the edges in the hyperbox that we are summing over has length 0, the sum over algebra elements in the hypercubical collection when playing the song will be empty, if $k \geq 2$.  This is true because when $k \geq 2$, the $k$th standard symphony contains a harmony with the element $i$ at least once.  According to the definition of playing a song, and thus in compression, in order for there to be nonzero terms in the formula, the number of harmonies that contain $i$ must be at most $d_i$; however, we have established that this is 0.  Therefore, the destabilization for $+K \cup \vec{M'}$ must be 0.
\end{proof}

\begin{lemma} \label{identitybox}
For a basic system, the map $\Gamma^{+K_i}$ is contraction by the dual of $x^1$ in $H^1(Y)^*$, after the appropriate identifications with $HC^\infty$.
\end{lemma}
\begin{proof}
We prove this with $i=1$.  We want to study $\Gamma^{+K_1}:00* \rightarrow 10*$.  Let's return to the chain level for now.  First of all, by Lemma~\ref{higherhomotopiesvanish} $\Gamma^{+K_1} = \mathcal{S}^{+K_1} = \Phi^{+K_1}$ in a basic system.    Now, we may identify the complex $00* \rightarrow 10*$ with the complex $01* \rightarrow 11*$ by applying $\Gamma^{+K_2} = \Phi^{+K_2}$, which we can think of as the complex $0* \rightarrow 1*$ for $L' \cup K_1$.

However, by construction, the generators of the $E_3$ pages in $0*$ (exterior elements with $x^1$ in them) were defined by applying $(\Phi^{+K_1}_*)^{-1}$ to the corresponding element in $1*$ in Lemma~\ref{d3calc} (with no $x^1$).  Therefore, under these identifications on the $E_3$ page, $\Gamma^{+K_1}$ is given by simply removing $x^1$.  Thus, passing to the homology of $0*$ and $1*$, the $E_4=E_\infty$ page by induction, the induced maps also behave the same way.
\end{proof}

By ignoring the ``$x^i \wedge$" components of $00*$, we will in fact think of $\Gamma^{+K_i}$ as the identity.  It turns out that knowing $\Gamma^{+K_i}$ is exactly what we need to understand $\Gamma^{-K_i}$ via our inductive arguments.

\begin{lemma}
With the basic system and the corresponding identifications as above, the map $\Gamma^{-K_i}$ is given by $Id + (d^{K_i}_3)_*$.
\end{lemma}
\begin{proof}
Again, we assume $i=1$.  By applying $\mathcal{S}^{+K_2} = \Phi^{+K_2}$ as above, we can prove the result on $0* \rightarrow 1*$ instead.  Now, by induction, $\hfboldinfty(S^3_{\mathbf{0}}(L' \cup K_1)) \cong \hcinfty(S^3_{\mathbf{0}}(L' \cup K_1))$ and therefore the higher differentials after $d_3$ vanish.  We consider the entire $E_3$ page for surgery on $L' \cup K_1$.  However, we induce a new filtration on this, $\tilde{\F}(x) = -\varepsilon_1$.  Since this filtration has depth 1, to understand the spectral sequence that this filtration gives, it suffices to calculate the homology of the associated graded and then $d_1$.

Let us do a quick algebra review first.  Consider a chain map $f:C_1 \rightarrow C_2$.  Now, we filter $M(f)$ by
\begin{equation*}
O(x) = \left\{
\begin{array}{rl}
1 & \text{if } x \in C_1,\\
0 & \text{if } x \in C_2.
\end{array} \right.
\end{equation*}
We can explicitly describe the pages in the spectral sequence arising from $O$.  The $E_0$ term splits as $C_1 \oplus C_2$.  The $E_1$ term will be given by $H_*(C_1) \oplus H_*(C_2)$.  Now, $d_1$ will be given by $f_*$.  Finally, all higher differentials will vanish.

On the chain level, $\hfboldinfty(S^3_{\mathbf{0}}(L' \cup K_1)$ is quasi-isomorphic to the mapping cone of $\Gamma^{+K_1} + \Gamma^{-K_1}:0* \rightarrow 1*$.  Now, split $d_3$ as $(d_3 - d^{K_1}_3) + d^{K_1}_3$.  With respect to the spectral sequence coming from $\tilde{\F}$, we see the only component of the differential that preserves the filtration level is $d_3 - d^{K_1}_3$.  Therefore, the associated graded splits as $H_*(0*)$ and $H_*(1*)$.  The differential, $d^{\tilde{\F}}_1$, is thus given by $\Gamma^{+K_1}_* + \Gamma^{-K_1}_*$.  However, we also know that $(d^{K_1}_3)_*$ must be $d^{\tilde{\F}}_1$.  Applying Lemma~\ref{identitybox} gives the desired result.
\end{proof}

\section{The Final Calculation}
Recall that we are interested in the calculation of
\begin{equation}
\Psi^{K_1,K_2} = \Gamma^{+K_1} + \Gamma^{-K_1} \circ (\Gamma^{-K_2})^{-1} \circ \Gamma^{+K_2} : 00* \rightarrow 10*,
\end{equation}
after we have taken homology.  By our identifications of the previous section, this corresponds to studying
\begin{equation} \label{gammas}
Id + (Id+d^{K_1}_3)_* \circ (Id+d^{K_2}_3)^{-1}_* \circ Id,
\end{equation}
where $Id$ secretly means contracting out the $x^1$ component of the exterior algebra elements.

\begin{lemma} \label{finalgammaslemma}
We have the following equality of $\fpowers$-module maps:
\begin{equation*} \label{finalgammas}
(\ref{gammas}) = (d^{K_1}_3)_* + (d^{K_2}_3)_* + (d^{K_1}_3)_* \circ (d^{K_2}_3)_*  =  (d^K_3)_* + (d^{K_1}_3)_* \circ (d^{K_2}_3)_*
\end{equation*}
\end{lemma}
\begin{proof}
First, we note that $d^{K_i}_3 = d_3 - (d_3 - d^{K_i}_3)$.  Clearly $(d_3)^2=0$.  However, $d_3 - d^{K_i}_3$ is the differential on $\cinfty(S^3_{\mathbf{0}}(L'))$, so this squares to 0 as well.  Therefore, $(d^{K_i}_3)^2 = 0$.  It is now easy to see that
\begin{equation*}
(Id + (d^{K_i}_3)_*)^2 = Id + ((d^{K_i}_3)^2)_* = Id.
\end{equation*}

Thus, $(Id + (d^{K_i}_3)_*)^{-1} = Id + (d^{K_i}_3)_*$.  Furthermore, we have that (\ref{gammas})$ = (d^{K_1}_3)_* + (d^{K_2}_3)_* + (d^{K_1}_3)_* \circ (d^{K_2}_3)_*$.  However, by the construction that $K = K_1 \# K_2$, $d^K_3 = d^{K_1}_3 + d^{K_2}_3$; thus, the corresponding relation on homology holds.
\end{proof}

We now want to show that adding this cross-term cannot remove any of the kernel of $(d^K_3)_*$.  This will be enough to calculate $\hfboldinfty$.

\begin{proposition} \label{finalproposition} The kernel of $(d^K_3)_*$ is contained in the kernel of (\ref{gammas}).
\end{proposition}
\begin{proof}
Suppose that $(d^K_3)_*(x)$ is 0.  This implies $(d^{K_1}_3)_*(x) = (d^{K_2}_3)_*(x)$ by Lemma~\ref{finalgammaslemma}.  We want to see that $(d^{K_1}_3)_* \circ (d^{K_2}_3)_*(x) = 0$.  Simply calculate,
\begin{equation*}
(d^{K_1}_3)_* \circ (d^{K_2}_3)_* (x)= (d^{K_1}_3)_* \circ (d^{K_1}_3)_* (x) = ((d^{K_1}_3)^2)_*(x) = 0. \qedhere
\end{equation*}
\end{proof}

\begin{proof}[Proof of Theorem~\ref{connectsumimplies}] The key observation is that Proposition~\ref{finalproposition} implies that the rank of $\hfboldinfty(S^3_{\mathbf{0}}(L))$ is at least that of $HC^\infty(S^3_{\mathbf{0}}(L))$.
\end{proof}

\bibliographystyle{amsplain}

\bibliography{biblio}

\end{document}